\documentclass[10pt,a4paper]{amsart}
\usepackage{bbm}
\usepackage{mathtools,amssymb,amsmath,amsopn,amsthm,caption,graphicx,centernot,cancel}
\usepackage{mathrsfs} 
\usepackage{tikz}
\usepackage{xcolor}
\usepackage{enumitem}
\usepackage[normalem]{ulem} 
\usepackage{verbatim}
\definecolor{darkgreen}{rgb}{0,0.5,0}
\usepackage[citecolor=darkgreen]{hyperref}
\usepackage{tikz-cd}
\usepackage{stmaryrd} 
\usepackage{colonequals}
\usepackage[OT2,T1]{fontenc}
\usepackage{titlesec} 
\titleformat{\section}{\normalfont\bfseries\filcenter}{\thesection}{1em}{}
\titleformat{\subsection}{\normalfont\bfseries}{\thesubsection}{1em}{}
\titleformat{\subsubsection}{\normalfont\bfseries}{\thesubsubsection}{1em}{}
\usetikzlibrary{cd}
\hypersetup{
colorlinks=true,
linkcolor=blue,
filecolor=magenta,
urlcolor=cyan,}
\usepackage[alphabetic, backrefs, lite]{amsrefs} 
\usetikzlibrary{matrix}

\usepackage[all]{xy}

\def\Alphabet{A,B,C,D,E,F,G,H,I,J,K,L,M,N,O,P,Q,R,S,T,U,V,W,X,Y,Z}
\def\alphabet{a,b,c,d,e,f,g,h,i,j,k,l,m,n,o,p,q,r,s,t,u,v,w,x,y,z}
\def\endpiece{xxx}
\def\makeAlphabet[#1]{\expandafter\makeA#1,xxx,}
\def\makealphabet[#1]{\expandafter\makea#1,xxx,}
\def\makeA#1,{\def\temp{#1}\ifx\temp\endpiece\else%
\mkbb{#1}\mkfrak{#1}\mkbf{#1}\mkcal{#1}\mkscr{#1}\mkbs{#1}\expandafter\makeA\fi}%
\def\makea#1,{\def\temp{#1}\ifx\temp\endpiece\else\mkfrak{#1}\mkbf{#1}\mkbs{#1}\expandafter\makea\fi}%
\def\mkbb#1{\expandafter\def\csname bb#1\endcsname{\mathbb{#1}}}
\def\mkfrak#1{\expandafter\def\csname fr#1\endcsname{\mathfrak{#1}}}
\def\mkbf#1{\expandafter\def\csname b#1\endcsname{\mathbf{#1}}}
\def\mkcal#1{\expandafter\def\csname c#1\endcsname{\mathcal{#1}}}
\def\mkscr#1{\expandafter\def\csname s#1\endcsname{\mathscr{#1}}}
\def\mkbs#1{\expandafter\def\csname bs#1\endcsname{{\boldsymbol{#1}}}}
\def\makeop[#1]{\xmakeop#1,xxx,}
\def\mkop#1{\expandafter\def\csname #1\endcsname{{\mathrm{#1}}}} %
\def\xmakeop#1,{\def\temp{#1}\ifx\temp\endpiece\else\mkop{#1}\expandafter\xmakeop\fi}%
\def\makeup[#1]{\xmakeup#1,xxx,}
\def\mkup#1{\expandafter\def\csname #1\endcsname{{\mathrm{#1}\,}}} %
\def\xmakeup#1,{\def\temp{#1}\ifx\temp\endpiece\else\mkup{#1}\expandafter\xmakeup\fi}%
\makeAlphabet[\Alphabet]
\makealphabet[\alphabet]
\makeop[Hom,Tor,Ker,Im,Coker,holim,hocolim,dim,GL,H,R,Gal,inv,ord,lcm,gcd,res,C,D,Div,loc,Princ,loc,Supp,cor,sh,tors,Tr,Cas]
\makeup[Spec,Proj,Spwf,Sp,Sh,Spf,Sch,id,dR,rig,cris,Tot,sm,an,Pic,NS,Br,Aut]
  \makeatletter
    
    \@addtoreset{equation}{section}
  \makeatother

\newcommand*{\sheafhom}{\cH\kern -.5pt om}

\newcommand{\Gm}{\bbG_{\mathrm{m}}}

\newcommand{\Sel}{\mathrm{S}}

\newcommand{\To}{\longrightarrow}
\newcommand{\eps}{\varepsilon}

\newcommand{\Z}{\bZ}
\newcommand{\Q}{\bQ}
\newcommand{\cs}{\chi(\sigma)}
\newcommand{\ct}{\chi(\tau)}

\newcommand{\csd}{\chi'(\sigma)}

\newcommand{\cv}[1]{\widehat{#1}}
\newcommand\restr[2]{{
  \left.\kern-\nulldelimiterspace 
  #1 
  \vphantom{\big|} 
  \right|_{#2} 
  }}
  
  \newcommand{\ctp}{\mathrm{CT}}
\newcommand{\ch}{\mathrm{char}}
\newcommand{\zz}{\mathrm{Z}}

\newcommand{\orbs}{\mathrm{orbits}}
\DeclareSymbolFont{cyrletters}{OT2}{wncyr}{m}{n}
\DeclareMathSymbol{\Sha}{\mathalpha}{cyrletters}{"58}
\renewcommand{\div}{\operatorname{div}}

\theoremstyle{plain}
    \newtheorem{theorem}{Theorem}[section]
    \newtheorem{proposition}[theorem]{Proposition}
    \newtheorem{lemma}[theorem]{Lemma}
    \newtheorem{corollary}[theorem]{Corollary}

      \newtheorem{proposition-definition}[theorem]{Proposition-Definition}
\theoremstyle{definition}
    \newtheorem{definition}[theorem]{Definition}

    \newtheorem{remark}[theorem]{Remark}

\newcommand{\note}[1]{\textcolor{red}{\sf[#1]}}

\begin{document}
\author[Shukla, H.]{Himanshu Shukla}
\thanks{HS was supported by Deutsche Forschungsgemeinschaft (DFG)-Grant STO 299/17-1 during this research.}
\address{Mathematisches Institut, Universit\"{a}t Bayreuth, Germany}
\email{\href{Himanshu.Shukla@uni-bayreuth.de}{Himanshu.Shukla@uni-bayreuth.de}}

\author[Stoll, M.]{Michael Stoll}
\address{Mathematisches Institut, Universit\"{a}t Bayreuth, Germany}
\email{\href{Michael.Stoll@uni-bayreuth.de}{Michael.Stoll@uni-bayreuth.de}}

\title[The Cassels-Tate pairing on 2-Selmer groups]%
      {The Cassels-Tate pairing \\ on 2-Selmer groups of elliptic curves}

\begin{abstract}
	We explicitly compute the Cassels-Tate pairing on the 2-Selmer
	group of an elliptic curve using the Albanese-Albanese definition of the
	pairing given by Poonen and Stoll. This leads to
	a new proof that a pairing defined by Cassels on the 2-Selmer groups of elliptic curves
	agrees with the Cassels-Tate pairing.
\end{abstract}

\maketitle


\section{Introduction}

Let $E$ be an elliptic curve over an algebraic number field~$K$. Then
according to the Mordell-Weil theorem, the set~$E(K)$ of $K$-rational
points on~$E$ is a finitely generated abelian group, so that
$E(K)\cong E(K)_{\tors}\oplus \Z^{r_{E/K}}$
with finite torsion group $E(K)_{\tors}$ and an integer~$r_{E/K} \ge 0$, which
is called the \emph{algebraic rank} or the \emph{Mordell-Weil rank} of $E/K$.
It is then a natural problem to determine the group structure
of~$E(K)$ for a given elliptic curve~$E/K$.
In practice, determining the torsion subgroup $E(K)_{\tors}$ for $E/K$ is not hard
(see \cite{silverman1}*{Theorem VII.3.4}). On the other hand, so far there
is no algorithm known that provably determines the rank~$r_{E/K}$.

By exhibiting independent points in~$E(K)$, one obtains a lower bound on the rank.
It is also possible to obtain upper bounds on~$r_{E/K}$, as follows.
Fix an integer $n \ge 2$. The short exact sequence of $K$-Galois modules
\[ 0 \To E(\bar{K})[n] \To E(\bar{K}) \stackrel{{}\cdot n}{\To} E(\bar{K}) \To 0 \]
induces, via the long exact sequence in Galois cohomology, a short exact sequence
\[ 0 \To E(K)/nE(K) \To \H^1(K, E(\bar{K})[n]) \To \H^1(K, E(\bar{K}))[n] \To 0 \,. \]
We can apply the same construction with $K$ replaced by any of its completions~$K_v$,
where $v$ runs through the places of~$K$. We then obtain the following commutative diagram
with exact rows.
\[ \xymatrix{0 \ar[r]
              & \dfrac{E(K)}{nE(K)} \ar[r] \ar[d]
							& \H^1(K, E(\bar{K})[n]) \ar[r] \ar[d] \ar[dr]^{\alpha}
							& \H^1(K, E(\bar{K}))[n] \ar[r] \ar[d]
							& 0 \\
						 0 \ar[r]
						  & \displaystyle\prod_v \dfrac{E(K_v)}{nE(K_v)} \ar[r]
						  & \displaystyle\prod_v \H^1(K_v, E(\overline{K_v})[n]) \ar[r]
						  & \displaystyle\prod_v \H^1(K_v, E(\overline{K_v}))[n] \ar[r]
						  & 0}
\]
It follows that $E(K)/nE(K)$ maps injectively into the
\emph{$n$-Selmer group} of~$E/K$, $\Sel^{(n)}(E/K) \coloneqq  \ker(\alpha)$.
More precisely, there is a short exact sequence
\begin{equation} \label{eqn:descent_sequence}
  0 \To E(K)/nE(K) \To \Sel^{(n)}(E/K) \to \Sha(E/K)[n] \To 0 \,,
\end{equation}
where
\[ \Sha(E/K) \coloneqq  \ker\bigl(\H^1(K, E(\bar{K})) \to \prod_v \H^1(K_v, E(\overline{K_v}))\bigr) \]
is the \emph{Shafarevich-Tate group} of~$E/K$.

The relevant fact here is that $\Sel^{(n)}(E/K)$ is \emph{finite} and
(at least in principle; see~\cite{stosurvey}) \emph{computable}. Since
\[ \#E(K)_{\tors}[n] \cdot r_{R/K}^n = \#\bigl(E(K)/nE(K)\bigr) \le \#\Sel^{(n)}(E/K) \,, \]
this gives, for each~$n$, an upper bound on~$r_{E/K}$, which is
sharp if and only if $\Sha(E/K)[n] = 0$.
For details on various interpretations of Selmer groups and their effective
computation see also \cite{cfoss1} and~\cite{effective_p_descent}.

It is a standard conjecture that $\Sha(E/K)$ is finite. This would imply
that for suitable~$n$, the bound will indeed be sharp. By work of
Kolyvagin, Wiles, and others, finiteness of~$\Sha(E/K)$ is known when
$K = \Q$ and the \emph{analytic rank} of~$E$, which is the order of
vanishing of the Hasse-Weil $L$-function of~$E$ at $s = 1$, is at most~$1$;
in this case the algebraic rank agrees with the analytic rank (this is
the weak Birch and Swinnerton-Dyer conjecture). But in general, the
conjecture is wide open.

It is easy to see that the natural map $E[n^2] \to E[n]$ induces a homomorphism
$s_n \colon \Sel^{(n^2)}(E/K) \to \Sel^{(n)}(E/K)$, which fits into a commutative
diagram as follows.
\[ \xymatrix{\dfrac{E(K)}{n^2 E(K)} \ar[r] \ar[d] & \Sel^{(n^2)}(E/K) \ar[d]^{s_n} \\
						 \dfrac{E(K)}{n E(K)} \ar[r] & \Sel^{(n)}(E/K)}
\]
This shows that the image of $E(K)/nE(K)$ in~$\Sel^{(n)}(E/K)$ is contained
in the image of~$s_n$. So we may be able to improve our bound on the rank
if we can determine this image. This can be done by computing the $n^2$-Selmer
group. Alternatively, we can use the \emph{Cassels-Tate pairing}. This is
an alternating pairing
\[ \langle {\cdot}, {\cdot} \rangle_\ctp \colon \Sha(E/K) \times \Sha(E/K) \to \Q/\Z \,, \]
which is perfect after dividing out the infinitely
divisible subgroup (which conjecturally is trivial; see above).
We will use the abbreviation CTP to refer to the pairing.
By pull-back, we obtain a pairing on~$\Sel^{(n)}(E/K)$, which we will also call~CTP.
The properties of the CTP imply that the image of~$s_n$ is the same
as the (left or right) kernel of the CTP on~$\Sel^{(n)}(E/K)$.
So if we can compute the CTP on the $n$-Selmer group, then we
can get a bound on~$r_{E/K}$ that is equivalent to the bound
obtained from the $n^2$-Selmer group.

See~\cite{psalb} for detailed information on the CTP (which can
be defined more generally on $\Sha(A/K) \times \Sha(A^\vee/K)$, where
$A/K$ is an abelian variety with dual abelian variety~$A^\vee$).
In particular, this paper provides four different definitions of the
pairing (and shows that they are all equivalent), among them
the \emph{Weil-pairing} definition~\cite{psalb}*{\S 12.2},
the \emph{homogeneous space definition}~\cite{psalb}*{\S 3.1}
and the \emph{Albanese-Albanese definition}~\cite{psalb}*{\S 12.1}.

Cassels~\cite{cas98} defined a pairing on
$\Sel^{(2)}(E/K) \times \Sel^{(2)}(E/K)$ using explicit models for the
$2$-coverings representing elements of the $2$-Selmer group, which was
shown to be the same as the CTP in~\cite{yoga}.
This pairing of Cassels can be seen as lying  somewhere in between
the Weil-pairing and the homogeneous space definitions. Later Cassels's
approach was generalized by Swinnerton-Dyer~\cite{sd22m},
Fisher and Newton~\cite{fisnew} and Fisher and van Beek~\cite{fisvanbeek}
to compute the CTP on $\Sel^{(2)}(E/K) \times \Sel^{(2^m)}(E/K)$,
$\Sel^{(3)}(E/K) \times \Sel^{(3)}(E/K)$ and for Selmer groups of
isogenies of odd prime degrees, respectively.
Using an approach of Donnelly based on the
homogeneous space definition of the CTP,
Fisher~\cite{fisdonsimp} has given another way of computing the CTP
on $\Sel^{(2)}(E/K)\times \Sel^{(2)}(E/K)$. Fisher also has an approach based
on the homogeneous space definition for $\Sel^{(3)}(E/K) \times \Sel^{(3)}(E/K)$ (unpublished).

The aim of this article is to compute the CTP on $\Sel^{(2)}(E/K)\times \Sel^{(2)}(E/K)$
using the Albanese-Albanese definition. It was remarked in~\cite{psalb} that designing
an algorithm to compute the CTP on the Selmer groups associated to the Jacobian of a curve
via the Albanese-Albanese definition may be easier than via the homogeneous space
or Weil pairing definitions, since it involves
dealing only with divisors on the curve; to the best of our knowledge this article is
the first attempt to do so.

We introduce two technical innovations. The first is to extend the two
Galois-equivariant pairings (between principal divisors and degree zero
divisors) used in the Albanese-Albanese definition to make them everywhere
defined, rather than only on divisors with disjoint support. This helps us
avoid some complications which arise from the assumption used in the original
definition, which demands that the lifts of certain elements of~$\Pic^0(E)$ to
$\Div^0(E)$ have disjoint support.
The second innovation is to write an element $a \in \Sel^{(2)}(E)$ as a
sum of quantities in~$\H^1(K, E(\bar{K})[2])$ not necessarily belonging
to~$\Sel^{(2)}(E)$. Using this we construct some maps (that may depend
on choices made) that mimic the construction of the CTP, and show that
the CTP is the sum of these functions (see Section~\ref{cor_method}).
By making the choices carefully, we obtain a fairly simple proof of
the equivalence of the CTP with the pairing defined by Cassels in~\cite{cas98}.

Building on the methods of this article, the
first author has developed an algorithm to compute the CTP on
$\Sel^2(J) \times \Sel^2(J)$, where $J$ is the Jacobian
of a hyperelliptic curve having an odd-degree model.
This will be described in a forthcoming article~\cite{ctpoddhyp}.

The organization of the paper is as follows.
Sections 2 and~3 provide the necessary notations and preliminaries.
In Section~4, we compute the pairing in a special case, and then show
in Section~5 that the general case follows from the special case.
In the same section we also show that the formula for the
CTP obtained using the Albanese-Albanese definition agrees with
the definition given by Cassels.


\section{Notation} \label{notations}

For a topological group $G$ and a $G$-module $A$, let $\partial$, $\cC^{n}(G,A)$, $\cZ^n(G,A)$, $\cB^n(G,A)$ and
$\H^n(G,A)$ denote the coboundary map, and the group of continuous $n$-cochains, $n$-cocycles,
$n$-coboundaries
and $n$-cohomology classes
respectively, with respect to
the bar resolution.
For a subgroup $H$ of $G$, let $\res^H_G$ and $\cor^G_H$ denote the restriction and
corestriction maps respectively at the level of cochains.
We will drop the subscript
and superscript in $\res$ and $\cor$ when clear from the context in order to simplify the
notations. Note that for $\cor$ to make sense $H$ has to be a finite index
subgroup of $G$.

Let $\bar{K}$ and $G_K$ be an
algebraic closure of a
perfect field
$K$ and the Galois group $\Gal(\bar{K}/K)$ respectively, and if $K$ is a global
field, then for each place $v$ of
$K$ let $K_v$ denote
the completion of $K$ with respect to $v$.
Let $\Gm\coloneqq \bar{K}^\times$ as a Galois module. We fix an embedding
$\bar{K}\hookrightarrow \overline{K_v}$; this induces 
the embedding $G_{K_v}\hookrightarrow G_K$ of absolute Galois groups. 
Let $A_v$ be the $G_{K_v}$-module obtained from a $G_K$-module
$A$ via $G_{K_v}\hookrightarrow G_{K}$. We will drop the
subscript from $A_v$ whenever
clear from the context.
For a cochain $x\in\cC^i(G_K,-)$ we will denote its restriction to
$\cC^i(G_{K_v},-)$ by $x_v$.
Furthermore, for a finite $G_K$-module
$A$, let $K(A)$ denote the smallest finite normal extension of $K$ such
that $G_{K(A)}$ acts trivially on~$A$.
To simplify notation, we will denote
$\cC^i(G_K,\Gm)$ by $\cC^i(K)$, and similarly for
cocycles and cohomology classes.
Let $\Br(K)\simeq \H^2(K)$ denote the  Brauer group, and  if $K$ is a number field, 
then for a place $v$
of $K$, let $\inv_{K_v} \colon \Br(K_v)\to\Q/\Z$
denote the local invariant map, and for $x,y\in K^\times$ we denote the 
quadratic Hilbert symbol of $x$ and $y$ with respect to $K$ by $(x,y)_K$.
If $K$ is a local field, then we have that $(-1)^{2\inv_K([(x,y)])}=(x,y)_K,$,
where $[(x,y)]$ denotes
the class of the quaternion algebra $(x,y)$ in $\Br(K)$.

We assume that $\ch(K) \ne 2$. Let $E/K$ be an
elliptic curve given by the equation
\begin{equation*} 
Y^2=f(X)\coloneqq X^3+AX+B,
\end{equation*}
where $A,B\in K$.
Let $e_1,e_2,e_3\in \bar{K}$ be the roots of $f$. Define $T_i\coloneqq (e_i,0)$ for $1\leq i\leq 3$, $T_0\coloneqq (0:1:0)$ (the
 point at infinity on $E$),
and $\Delta\coloneqq  E[2]\setminus\{T_0\}$.
Note that $\Delta$ is a
$G_K$-set, i.e., a set with
$G_K$ action. For any $G_K$-set $S$ and a $G_K$-module $M$
we define the $G_K$-module~$M^S$ of maps $S \to M$. Its elements can be
interpreted as formal sums of the form $\sum\limits_{s\in S} a_s(s)$ with $a_s\in M$.
The group structure of~$M^S$ is defined point-wise and the $G_K$-action is naturally defined as
$\sigma\cdot\sum\limits_{s \in S} a_s(s)\coloneqq  \sum\limits_{s \in S}\sigma a_s(\sigma s)$, for $\sigma\in G_K$.
In view of the above one can show that
$$\mu_2^\Delta \cong E[2]\oplus \langle(-1)(T_1)+(-1)(T_2)+(-1)(T_3)\rangle;$$
the inclusion $E[2] \hookrightarrow \mu_2^\Delta$ is induced by the Weil pairing,
$P \mapsto \sum_{T \in \Delta} e_2(P, T) (T)$.
Let $A\coloneqq  K[X]/\langle f(X)\rangle$ denote the \'{e}tale algebra
associated to $\Delta$. Then we have
\begin{equation}\label{h1_group_representation}
 \H^1(G_{K},E[2])\cong \mathrm{ker}
\left(N \colon A^{\times}/(A^{\times})^2\longrightarrow K^{\times}/(K^{\times})^2\right),
\end{equation}
where $N$ denotes the
map induced by the norm map
from $A$ to $K$.

Let $\bar{A}\coloneqq A\otimes_K \bar{K}$.
The elements of $A^\times/(A^\times)^2$ can be represented by elements
of~$A^\times$, which can be written in the form $\beta = l_0 + l_1 \theta + l_2 \theta^2$,
where $l_0, l_1, l_2 \in K$ and $\theta$ is the image of~$x$ in~$A$.
We set $\beta_i =l_0+l_1e_i+l_2e_i^2 \in K(e_i)^\times$.
Under the identification $A \hookrightarrow \bar{A} \cong \bar{K}^3$,
$\beta$ is then mapped to $(\beta_1, \beta_2, \beta_3)$.

If $\beta$ represents an element~$a$ of~$\H^1(G_K,E[2])$ via~\eqref{h1_group_representation},
then $\beta_1\beta_2\beta_3\in (K^\times)^2$.
Fix square-roots $\sqrt{\beta_i}$ of $\beta_i$, for
$1\leq i\leq 3$, and consider the formal sum $\sum\limits_{i=1}^3
\sqrt{\beta_i}(T_i)\in \bar{A}^\times=\Gm^\Delta$.
This element is a square root $\sqrt{\beta}$ of $\beta\in A^\times$ considered 
as an element of $\Gm^\Delta$. Since the 1-coboundary
$\sigma\mapsto\sigma\sqrt{\beta}/\sqrt{\beta}$ then takes values in $\mu_2^\Delta$ because 
$\beta$ is fixed by $G_K$, it represents a 1-cocycle
$\chi_a\in\cZ^1(G_K,\mu_2^\Delta)$.
The cocycle $\chi_a$ is associated to a 1-cocycle with values
in $E[2]$ representing $a$  via
 $(\Z/2\Z)^\Delta\to E[2]$ sending $\sum\limits_{i=1}^3 a_i(T_i)\mapsto \sum\limits_{i=1}^3 a_iT_i$ and the identification $\mu_2\cong \Z/2\Z$ as $G_K$-modules.

Note that $\mu_2^\Delta\simeq \bigoplus\limits_\orbs\mu_2^{\Delta_i}$, 
where $\Delta_i$ are the $G_K$-orbits of $\Delta$. This induces an isomorphism 
$\cZ^1(G_K,\mu_2^\Delta)\simeq\bigoplus\limits_\orbs\cZ^1(G_K,\mu_2^{\Delta_i})$, so
$\chi_a(\sigma)$ factors through the orbits of $\Delta$ i.e.
$\chi_a(\sigma)=\sum\limits_{\orbs}\chi_{a,i}(\sigma)$, where $\chi_{a,i}$ 
corresponds to $\chi_a$ when the support $\Delta$ is replaced by a 
$G_K$-orbit $\Delta_i$ of $\Delta$.
We associate with $\sum\limits_{i=1}^3a_i(T_i)\in\mu_2^\Delta$ the triple
$(a_1,a_2,a_3)\in\mu_2^3$, and
henceforth will use this
representation for the elements of $\mu_2^\Delta$.
Write $\hat{0}$ for the
triple $(1,1,1)$ and $\hat{i}$ for the triple with $1$ at $i^{\mathrm{th}}$
position
and $-1$ otherwise.
The action of $G_K$ on
the triples
in $\mu_2^3$ is
induced from the
action on $\mu_2^\Delta$.
If $(x_1,x_2,x_3)$ is a triple
representing an element of $(\mu_2^\Delta)^{G_K}$, then define the
product
$\prod\limits_{i}^\diamond x_i$ to be the product
taken over one representative~$i$ of each $G_K$-orbit on~$\Delta$
(note that $x_i = x_j$ when $i$ and~$j$ are in the same orbit).


\section{Preliminaries}\label{preliminaries}

From now on, we fix a number field~$K$ and an elliptic curve~$E/K$.
Recall the definitions of the Shafarevich-Tate and Selmer groups of~$E/K$
and the ``$n$-descent sequence''~\eqref{eqn:descent_sequence}
from the introduction.

We now recall some equations which will be useful later.
Let $P=(x,y)$ on $E/K$. Then
we have that
\begin{equation}\label{addTi}
	P+T_i=\left(\frac{e_ix+e_je_k-e_i(e_j+e_k)}{x-e_i},\frac{-(e_j-e_i)(e_k-e_i)y}{(x-e_i)^2}\right),
\end{equation}
where $\{i,j,k\} = \{1,2,3\}$.
Using the fact that
the line through $P$ and $T_i$ passes through $-P-T_i$ and equation
\eqref{addTi} we deduce that
\begin{equation}\label{xtieiandy}
\frac{x(\pm P+T_i)-e_i}{x(\pm P)-e_i}=\frac{(e_j-e_i)(e_k-e_i)}{(x-e_i)^2}=-\frac{y(\pm P+T_i)}{y(\pm P)}
\end{equation}
We also deduce from~\eqref{addTi} that
\begin{equation}\label{xpplustj-ei}
x(\pm P+T_j)-e_i = \frac{(x-e_k)(e_j-e_i)}{x-e_j},
\end{equation}
which along with the fact that the line passing through
$P+T_i$, and $T_k$ passes through
$-P+T_j$ gives
\begin{equation}\label{ypxpij}
\frac{y(\pm P+T_i)(x(\pm P)-e_k)}{y(\pm P)(x(\pm P+T_i)-e_k)}=\frac{(3B+2Ae_i)(x-e_k)}{e_i(x-e_i)(x-e_j)(e_i-e_k)}=\frac{(x-e_k)^2(e_j-e_i)}{y^2}.
\end{equation}


\subsection{Group cohomology}\label{grpcoho}

We now recall some results from group cohomology about $\cor$, $\res$ and the local invariant maps $\inv$.
Let $G$ be a group and let $H\leq  G$ be a finite index subgroup of~$G$. Let
$A$ be a $G$-module; then for $n\geq 0$ we have a homomorphism
called \emph{corestriction} on $n$-cochains:
\begin{equation}\label{cor1}
\cor^G_H \colon \cC^{n}(H,A)\to \cC^{n}(G,A).
\end{equation}
We use the explicit description of $\cor^G_H$ with respect to the standard
resolution given in
\cite{NSW2.3}*{\S I.5.4} to derive an explicit
description of $\cor^G_H$ with respect to the bar resolution.
Fix a set $R$ of representatives of right cosets, $Hg$, for $g\in G$.
We have a map $r \colon G \to R$, which maps $g$ to the representative of $Hg$
in $R$. This induces
a map $m \colon G\to H$ given by
$m(\sigma) = \sigma r(\sigma)^{-1}$.
Then using the
explicit
isomorphisms
between cochains with respect to
the bar resolution and
homogeneous cochains with respect to the standard
resolution we get
\begin{multline}\label{corbarres}
\left(\cor^G_H(\gamma)\right)(\sigma_1,\ldots,\sigma_n)\coloneqq
    \sum\limits_{g\in
    R}g^{-1}\gamma(m(g\sigma_1),m(g\sigma_1)^{-1}m(g\sigma_1\sigma_2),\ldots\\
    \ldots,m(g\sigma_1\ldots\sigma_{n-1})^{-1}m(g\sigma_1\ldots \sigma_n)).
\end{multline}

We recall some properties of $\cor$,
$\res$ and $\inv$ in the
following proposition.
\begin{proposition}[\cite{NSW2.3}*{Corollary 1.5.7, 7.1.4,  Proposition 1.5.3}]\label{rescor}
	Let $K\subset L\subset \bar{K}$ be a tower of fields with absolute
	Galois groups $G_K$, $G_L$ respectively,
	and $[L:K]<\infty$. For a $G_K$-module $A$ we have:
\begin{enumerate}
\item If $\gamma\in\H^{n}(G_K,A),$ then
    $\cor^{G_K}_{G_L}\circ\res_{G_K}^{G_L}(\gamma)= [L:K]\gamma$.
	\item If $K$ is a local field and $\gamma\in\H^{2}(K),$ then
$\inv_K\circ\cor^{G_K}_{G_L}(\gamma)=\inv_L(\gamma)$, and
		$\inv_L\circ\res^{G_L}_{G_K}(\gamma)=[L:K]\inv_K(\gamma)$.
			\item 	Let $\gamma\in \cC^{n}(G_K,A)$ and
                $\gamma'\in\cC^{m}(G_L,A')$, where $A$, $A'$ are $G_K$-modules. 
                Then we have
		\begin{equation}
		\cor^{G_K}_{G_L}(\res_{G_K}^{G_L}(\gamma)\cup\gamma')=\gamma\cup\cor^{G_K}_{G_L}(\gamma') \tag{whenever $\cup$ is defined}.
	\end{equation}
\end{enumerate}
\end{proposition}

We now recall the double coset formula.
Let $H$ be a subgroup of $G$ and $B$ be an $H$-submodule of a $G$-module $A$. For
$\sigma\in G$, the maps $H^\sigma\coloneqq \sigma H\sigma^{-1}\to H$ given by
$\tau\mapsto\sigma^{-1}\tau\sigma$ and $B\to \sigma B$ given by $b\mapsto \sigma
b$ form a compatible pair i.e. $\sigma ((\sigma^{-1}\tau\sigma)(b))
=\tau(\sigma(b))$, for all $b\in B$ and $\tau\in H^\sigma$.
Hence we have a well-defined map
$$\sigma_* \colon \cC^i(H,B)\to \cC^i(H^\sigma,\sigma B).$$
One can show that $\sigma_*$ commutes with $\inf$, $\cor$, $\res$, $\cup$ and $\partial$ \cite{NSW2.3}*{Proposition 1.5.3, 1.5.4}.
In particular, if $H=G$ and $B=A$, then using the fact that $\sigma_*$ commutes with dimension shifting, one can show that  $\sigma_*$ is the identity on $\H^i(G,A)$.
If $H$, $U$ are closed subgroups of $G$ with $[G:H]< \infty$, then we have the double coset formula which also holds at the level of cochains \cite{NSW2.3}*{Proposition 1.5.6}:
\begin{equation}\label{doublecosetformula}
\res^U_G\cor^G_H(z)=\sum\limits_{g\in R}\cor^U_{U\cap g H g^{-1}}\res^{U\cap g Hg^{-1}}_{g H g^{-1}} g_*(z),
\end{equation}
where $R$ is a system of double coset representatives $G=\bigsqcup\limits_{g\in R} UgH$ and $z\in \cC^i(H, A)$ ($\sqcup$ denotes the disjoint union).

Recall that for each place $v$ of $K$ we fixed an embedding 
$\bar{K}\hookrightarrow \overline{K_v}$.
Via this embedding, $G_{K_v}$ can be considered as a closed subgroup of $G_K$ 
($G_K$ is Hausdorff and $G_{K_v}$ is compact). Let $L$ be a finite extension of
$K$ and $w_1,\ldots, w_m$ be all the distinct places  of $L$ above $v$ with 
$w_1$ being the one induced by the fixed embedding $\bar{K}\hookrightarrow
\overline{K_v}$.
There is a $g_i\in G_K$ corresponding to each $w_i$ such that $w_i$ is induced
by the composite embedding $g_i:\bar{K}\to \bar{K}\hookrightarrow \overline{K_v}$. If 
$L=K(\theta_1)$, then $g_i$ correspond to those embeddings of $L\hookrightarrow
\bar{K}$ 
which map $\theta_1$ to one of its $G_K$-conjugates that is not a $G_{K_v}$-conjugate.
This implies that $\{g_1, \ldots, g_m\}$ is a system of double coset representatives of $G_K$
with respect to $G_{K_v}$ and $G_L$, i.e., $G_K=\bigsqcup\limits_{i=1}^m G_{K_v}g_iG_L$. Further, note that
$G_{L_{w_i}}\subset G_{K_v}$ fixes the extension $g_iL$ and hence
$G_{L_{w_i}}=G_{K_v}\cap g_iG_Lg_i^{-1}$. 
Using the double coset formula \eqref{doublecosetformula} we have the following remark:
\begin{remark}
Let $K$, $L$, $w_i$, $g_i$ be as above. Then
\begin{equation} \label{doublecosetfornumfields}
(\cor^{G_K}_{G_L}(z))_v=\res^{G_{K_v}}_{G_K}\circ\cor^{G_K}_{G_L}(z)=\sum\limits_{i=1}^m 
    \cor^{G_{K_v}}_{G_{L_{w_i}}}\circ\res^{G_{L_{w_i}}}_{G_{g_iL}} \circ(g_i)_*(z).
\end{equation}
\end{remark}

The following two subsections recall two
definitions of the CTP. We first consider the definition
by Cassels in~\cite{cas98} and then the Albanese-Albanese
definition in~\cite{psalb}. We continue to assume that $K$ is a number
field.


\subsection{Cassels's pairing}\label{seccaspairing}

Let $a,a'\in \Sel^{(2)}(E)$ be represented
by
$\beta = (\beta_1,\beta_2,\beta_3)$ and $\beta' = (\beta_1',\beta_2',\beta_3')$
as discussed in Section~\ref{notations}, and
for pairwise distinct $1\leq i,j,k\leq 3$, define quadratic
forms in variables $(U_1,U_2, U_3, T)$ by 
\begin{equation}\label{hi}
	H_i(U_1,U_2, U_3, T)\coloneqq (\beta_j\Gamma_j^2-\beta_k\Gamma_k^2)/(e_j - e_k) + T^2,
\end{equation}
where $\Gamma_i\coloneqq U_1+U_2e_i+U_3e_i^2$.
Note that $(U_1,U_2,U_3)\mapsto (\Gamma_1,\Gamma_2,\Gamma_3)$ is just a linear change of
coordinates.
Also note that each $H_i$ is defined over 
$K(e_i)$.
Any two of these quadratic forms define a projective curve $D_{a}$
in $\mathbb{P}^3$ with coordinates $U_1,U_2,U_3,T$.
Choosing $j,k\in\{1,2,3\}$ cyclically 
for a given $i\in\{1,2,3\}$ we get 
$\sum\limits_{i=1}^3(e_j-e_k)H_i=0$. 
Hence any two of the $H_i$ define the same curve $D_a$ over $K$.
$D_a$ has points 
for every completion $K_v$ of $K$ and is a 2-covering of $E$ representing $a$. 
For details see \cite[\S 2]{cas98}.

Since $D_{a}$ has a point on
the affine patch $T\ne 0$ locally everywhere, each $H_i$ has a 
non-trivial solution over $K_v(e_i)$
for every place $v$ of $K$, and therefore each $H_i$ has a solution over 
every completion of $K(e_i)$. This implies that for each $i$ there is a point
$\mathfrak{q}_i\coloneqq (\mathfrak{u}_1:\mathfrak{u}_2:\mathfrak{u}_3:1)$ or
$(\Gamma_j^*:\Gamma_k^*:1)$ defined
over $K(e_i)$ satisfying
$H_i=0$, which is a consequence of
the
local-global principle for quadratic forms.
Let $L_i(U_1,U_2,U_3,T)$ for $1\leq i\leq 3$ be a linear form such that $L_i=0$ is
 the tangent to $H_i$ at $\mathfrak{q}_i$.

If $\mathfrak{q}_v$ is a point defined on $D_{a}$ over $K_v$, then
 Cassels' pairing \cite{cas98}*{Lemma 7.4}
is defined as follows:
\begin{equation}\label{cas}
	\langle\alpha,\alpha'\rangle_{\mathrm{Cas}}\coloneqq
    \prod\limits_{v}\prod\limits_i^\diamond(L_i(\mathfrak{q}_v),\beta_i')_{K_v(e_i)}.
\end{equation}
In \cite{cas98} Cassels showed that
the above definition
gives a well-defined pairing and is independent of the choices made.


\subsection{The Albanese-Albanese definition of the Cassels-Tate pairing}\label{defctp}

We now explain the \emph{Albanese-Albanese} definition
of the Cassels-Tate pairing on $\Sha(E)$ for an elliptic curve $E/K$, using
the notations
from Section~\ref{notations}.
Let $\Princ(E)$ denote the group of principal divisors on $E(\bar{K})$. 
Choose uniformizers $t_P$ for $P\in E(\bar{K})$ such that the map
$P\mapsto t_P$ is Galois-equivariant, and consider the
Galois-equivariant pairings
\begin{align}
    &\langle\cdot,\cdot\rangle_1 \colon \left(\Princ(E)\times\Div^0(E)\right)\to \Gm 
    \\
   &\langle\cdot,\cdot\rangle_2 \colon \left(\Div^0(E)\times\Princ(E)\right)\to\Gm 
\end{align}
defined as follows.
$$\langle\div(f), D\rangle_1\coloneqq \prod\limits_{P}(ft_P^{-v_P(f)})(P)^{v_P(D)},$$
and
$$\langle D,\div(f)\rangle_2\coloneqq \prod\limits_{P}(-1)^{v_P(f)v_P(D)}(ft_P^{-v_P(f)})(P)^{v_P(D)}.$$
The above pairings are well-defined and extend
the partially defined pairings
in \cite{psalb}*{\S 3.2}.
In what follows, ``$\cup_i$'' will denote the cup product of cochains with respect to
	the pairing $\langle\cdot,\cdot\rangle_i$ for $i\in\{1,2\}$,
	and we will drop the subscripts in $\cup_i$ whenever it
	is clear from the context.
Next we show that the
two pairings defined above agree on the diagonal
$\Princ(E)\times\Princ(E).$
\begin{proposition}\label{paireqondiag}
Let $\langle\cdot,\cdot\rangle_1$ and $\langle\cdot,\cdot\rangle_2$ be as above. Then $\langle\cdot,\cdot\rangle_1=\langle\cdot,\cdot\rangle_2$ on $\Princ(E)\times\Princ(E).$
\end{proposition}
\begin{proof}
Define the
 \emph{tame symbol} at $P$ for two nonzero functions
$f$ and $g$ as follows:     
$$[f,g]_P\coloneqq  (-1)^{v_P(f)v_P(g)}\frac{f^{v_P(g)}}{g^{v_P(f)}}(P).$$
We have
\begin{align*}
	\frac{\langle\div(f),\div(g)\rangle_{1}}{\langle\div(f),\div(g)\rangle_{2}}&=\prod\limits_P(-1)^{v_P(f)v_P(g)}\frac{\left(ft_P^{v_P(f)}\right)(P)^{v_P(g)}}{\left(gt_P^{v_P(g)}\right)(P)^{v_P(f)}}\\
	&=\prod\limits_P(-1)^{v_P(f)v_P(g)}\left(\frac{f^{v_P(g)}}{g^{v_P(f)}}\right)(P)=\prod\limits_P[f,g]_P=1.
\end{align*}
The last equality is a consequence of
strong Weil reciprocity~\cite{weilrec}.
\end{proof}

\begin{remark}
Note that one can define the above pairings for any smooth, projective,
    absolutely irreducible curve $C/K$ in the same fashion, and they will
    satisfy  Proposition \ref{paireqondiag}.
\end{remark}

Let $a,a'\in \H^1(G_K,E[n])$ be represented by 1-cocycles $\alpha,\alpha'$
with values in $E(\bar{K})[n]$.
Lift $\alpha,\alpha'$ to 1-cochains $\mathfrak{a}$ and $\mathfrak{a'}$
 with values in $\Div^0(E)$.
	Since $\partial\alpha =\partial\alpha'=0$, the cochains $\partial\mathfrak{a}$
	and $\partial \mathfrak{a}'$
take values in principal divisors as can be seen using Galois cohomology applied
to the
exact sequence
$$0\To\Princ(E)\To \Div^0(E) \To \Pic^0(E)\To 0.$$

	In what follows we will write the
group structure on the cochains/cohomology classes 
additively even when they take values 
in a multiplicative group.
However, after evaluation of cochains at certain arguments we will use the group operation of the respective group module.
For example, if $x,y\in \C^1(K)$, then
we use $+$ to denote their addition $z\coloneqq x+y$ as cochains, but for
$\sigma\in G_K$, $z(\sigma)\in\Gm$ will be written as $x(\sigma)y(\sigma)$ using the group operation of $\Gm$.

Define the 3-cocycle  $\eta\coloneqq \partial
\mathfrak{a}\cup_1\mathfrak{a}'-\mathfrak{a}\cup_2\partial\mathfrak{a}'\in\cZ^3(K)$.	
It is easy to verify that the
	cup products in $\eta$ make sense and that $\eta$ is indeed a cocycle.
By class
	field theory, we have
	$\H^{3}(K)=0$ for any number field $K$, hence
 $\eta=\partial\eps$ for some 2-cochain
	$\eps\in \mathcal{C}^2(K)$.

	Further, assume that $a\in \Sel^{(n)}(E)$, and
	let $v$ be a place
	of $K$.
	Then $\alpha_v = \partial\beta_v$ for some
	$\beta_v \in E(\overline{K_v})$, and we can lift $\beta_v$ to an element $\mathfrak{b}_v\in \Div^0(E)$.
	We then have that $\mathfrak{a}_v - \partial\mathfrak{b}_v$ takes values in
	principal divisors,
so we can define
the following 2-cocycle
	\begin{equation}
	\gamma_v = (\mathfrak{a}_v - \partial\mathfrak{b}_v) \cup_1 \mathfrak{a}'
	-\mathfrak{b}_v\cup_2\partial\mathfrak{a}'_v-\eps_v.
	\end{equation}
%
So $\gamma_v$
represents some
class $c_v \in \H^2(K_v)\simeq
\Br(K_v)$. We define the CTP on
	$\Sel^{(n)}(E)\times \H^1(G_K,E[n])$ as follows:
	\begin{definition}\label{def_CTP}
		Let $(a,a')\in\Sel^{(n)}(E)\times \H^1(G_K,E[n]),$
		and $c_v$ be as above. Then
		the Cassels-Tate pairing on
		$a,a'$ is defined as follows:
		\begin{equation}
			\langle a, a'\rangle_{\ctp}\coloneqq \sum\limits_{v}\inv_{k_v}(c_v).
		\end{equation}
	\end{definition}

\begin{remark}\label{ctpnotwd}
The above definition is well-defined up to the choice of
$\beta_v$ i.e. changing $\beta_v$
to $\beta_v+\kappa_v$ for some $\kappa_v\in E(K_v)$
changes the value of the pairing. However if
we assume that $a'\in \Sel^{(n)}(E)$, then $\langle a,a'\rangle_{\ctp}$ is
    independent of the choices made and hence is a well-defined pairing (see
    \cite{psalb}).
\end{remark}


\section{Computing CTP on $\Sel^2(E)\times \H^1(G_K, \langle
T_1\rangle)$}\label{modifiedparing}

In this section we
 compute CTP on $\Sel^2(E)\times\H^1(G_K,\langle T_1\rangle)$ (in the sense of \S \ref{defctp}) assuming
that $e_1\in K$ i.e. $[K(E[2]):K]\leq 2$.
Remark \ref{ctpnotwd} implies that
the value of CTP thus obtained depends on the choices made
during the local part of the computation. Therefore, it is one of the possible
values of CTP on $\Sel^{(2)}(E)\times\H^1(G_K,\langle T_1\rangle)$.
Henceforth, we will
always assume $1\leq i,j,k\leq 3$,
and if any subset of $i,j,k$ appear together in an expression, they will be  pairwise distinct.
 We use the notations from \S \ref{notations} during the process.

We begin with an explicit description of
$a\in\Sel^{(2)}(E)$ and $a'\in \H^1(G_K,\langle T_1\rangle)\simeq
K^\times/(K^\times)^2$, represented by the triple
$(\beta_1,\beta_2,\beta_3)$ and $\beta'\in K^\times$
respectively.
Let the 1-cocycles $\alpha$ and $\alpha'$ representing
$a$ and $a'$ respectively, be as follows:
\[
\alpha(\sigma)=\begin{cases}
	&T_0 \ \ \ \ \ \ \ \ \ \hfill{\cs=\cv{0}},\\
	\vspace{.1in}
	&T_i \ \ \ \ \ \ \hfill{\cs=\cv{i}},
\end{cases}
	\quad\mathrm{and}\quad
\alpha'(\sigma)=\begin{cases}
        &T_0 \ \ \ \ \ \ \ \ \ \hfill{\chi'(\sigma)=1},\\
        \vspace{.1in}
        &T_1 \ \ \ \ \ \ \hfill{\chi'(\sigma)=-1}.
\end{cases}
\]
Here $\chi=\chi_a$ (as defined in \S \ref{notations}), but we
have dropped the subscript for  simplicity
of notations.
Further, $\chi'(\sigma)\coloneqq \sigma(\sqrt{\beta'})/\sqrt{\beta'}$, for a fixed
square root $\sqrt{\beta'}$ of $\beta'$.

The next two subsections are dedicated to the computation of the CTP when $a,a'$
are as above.


\subsection{Global computation}\label{GlobComp}
Lift $\alpha,\alpha'$ to 1-cochains $\mathfrak{a}$, $\mathfrak{a}'$ with values in $\Div^0(E)$ as follows:
\[
\mathfrak{a}(\sigma)=\begin{cases}
	&0 \ \ \ \  \hfill{\cs=\cv{0}},\\
	\vspace{.1in}
	&(T_i)-(T_0) \ \ \  \hfill{\cs=\cv{i}},
\end{cases}
	\quad\mathrm{and}\quad
\mathfrak{a}'(\sigma)=\begin{cases}
        &0 \ \ \ \  \hfill{\chi'(\sigma)=1},\\
        \vspace{.1in}
        &(T_1)-(T_0) \ \ \  \hfill{\chi'(\sigma)=-1}.
\end{cases}
\]
We have:
\begin{equation}
\partial\mathfrak{a}(\sigma,\tau)=\begin{cases}
	&0=\div(1) \ \ \ \  \hfill{\cs=\cv{0}\ \mathrm{or}\ \ct=\cv{0}},\\
\vspace{.1in}
	&2(T_i)-2(T_0)=\div(x-e_i) \ \  \hfill{\cs=\cv{i},\ \sigma\cdot\ct=\cv{i}},\\
\vspace{.1in}
	&(T_i)+(T_j)-(T_k)-(T_0)=\div(\frac{y}{x-e_k}) \ \ \ \ \ \ \hfill{\cs=\cv{i},\ \sigma\cdot\ct=\cv{j}}.
\end{cases}
\end{equation}
Similarly, for $\mathfrak{a}'$ we have:
\begin{equation}
	 \partial\mathfrak{a}'(\sigma,\tau)=\begin{cases}
                &0=\div(1)\ \ \ \ \ \ \ \ \ \hfill{\chi'(\sigma)=1\ \mathrm{or}\ \chi'(\tau)=1},\\
        \vspace{.1in}
                &2(T_1)-2(T_0) =\div(x-e_1)\ \ \ \ \ \ \hfill{\chi'(\sigma)=\chi'(\tau)=-1}.
\end{cases}
\end{equation}
Let $t_P$ denote a unifomizer at point $P\in E(\bar{K})$.
We assume $t_{T_0}=x/y$, $t_{T_i}=-(x-e_i)/y$
and $t_P=x-x(P)$ at
all other points $P\notin E[2]$. The map defined by $P\mapsto t_p$
is Galois-equivariant. It is not hard to see that
$\langle \div(f),D\rangle_1=\langle D,\div(f)\rangle_2$ where $\div(f)$ and $D$ appear in
the values taken by $\partial\mathfrak{a},\partial\mathfrak{a}'$ and $\mathfrak{a}$, $\mathfrak{a}'$ respectively.
Therefore, we have:
\begin{align*}
	&\langle(x-e_i),(T_i)-(T_0)\rangle_1 = \frac{y^2/(x-e_i)(T_i)}
{(x-e_i)x^2/y^2(T_0)}=(e_i-e_j)(e_i-e_k),\\
	&\langle(x-e_i),(T_j)-(T_0)\rangle_1 = \frac{(x-e_i)(T_j)}
{(x-e_i)x^2/y^2(T_0)}=e_j-e_i,\\
	&\langle y/(x-e_k),(T_i)-(T_0)\rangle_1  =\frac{-y^2/(x-e_k)(x-e_i)(T_i)}
{x/(x-e_k)(T_0)}=e_j-e_i,\\
    &\langle y/(x-e_k),(T_k)-(T_0)\rangle_1 = \frac{-y(x-e_k)/y(x-e_k)(T_k)}
{x/(x-e_k)(T_0)}=-1.
\end{align*}
We set $s_{ij}\coloneqq e_i-e_j$ for $i\ne j$ and $s_i\coloneqq s_{ij}s_{ik}$.
For $\sigma, \tau, \rho \in G_K$, the cup product, $(\partial \mathfrak{a}\cup_1\mathfrak{a}')(\sigma,\tau,\rho)$
(resp. $(\mathfrak{a}\cup_2\partial\mathfrak{a}')(\sigma,\tau,\rho)$)
via the pairing $\langle\cdot,\cdot\rangle_1$ (resp. $\langle .,.\rangle_2$) is
$\langle\partial\mathfrak{a}(\sigma,\tau),\sigma\tau(\mathfrak{a}'(\rho))\rangle_1$
(resp. $\langle\mathfrak{a}(\sigma),\sigma(\partial\mathfrak{a}'(\tau,\rho))\rangle_2$)
using \cite{NSW2.3}*{Proposition 1.4.8}.
Therefore, we have
\begin{equation}
        (\partial\mathfrak{a}\cup_1\mathfrak{a}')(\sigma,\tau,\rho)=\begin{cases}
&1\ \ \ \ \ \ \ \ \ \hfill{\cs=\cv{0}\ \mathrm{or}\ \ct=\cv{0}\ \mathrm{or}\ \chi_1'(\rho)=1},\\
                \vspace{.1in}
		&s_1\ \ \ \hfill{\cs=\cv{1},\ \sigma\cdot\ct=\cv{1},\ \chi'(\rho)=-1},\\
                \vspace{0.1in}
		&s_{1j} \ \ \ \ \ \ \ \hfill{\cs=\cv{j},\ \sigma\cdot\ct=\cv{j},\ \chi'(\rho)=-1},\\
                \vspace{.1in}
		&s_{j1} \ \ \ \ \ \ \ \hfill{\substack{(\cs,\sigma\cdot\ct)=(\cv{1},\cv{j}),\ \mathrm{or}\\ (\cs,\sigma\cdot\ct)=(\cv{j},\cv{1})},\ \chi'(\rho)=-1},\\
                \vspace{.1in}
                &-1 \ \ \ \ \ \ \ \ \ \ \hfill\cs=\cv{j},\sigma\cdot\ct=\cv{k},\ \chi'(\rho)=-1,
        \end{cases}
\end{equation}
and
\begin{equation}
        (\mathfrak{a}\cup_2\partial\mathfrak{a}')(\sigma,\tau,\rho)=\begin{cases}
                &1\ \ \ \ \ \ \ \ \ \hfill{\cs=\cv{0}\ \mathrm{or}\ \chi'(\tau)=1\ \mathrm{or}\ \chi'(\rho)=1},\\
                \vspace{.1in}
                &s_1\ \ \ \hfill{\cs=\cv{1},\chi'(\tau)=\chi'(\rho)=-1},\\
                \vspace{.1in}
                &s_{j1} \ \ \ \ \ \hfill{\cs=\cv{j},\chi'(\tau)=\chi'(\rho)=-1}.
        \end{cases}
\end{equation}

We want to find a 2-cochain $\eps$ such that
$\partial\eps=\eta$, which (as we will later see) will require us to express 
certain elements as norms.
Since $D_a$ is locally everywhere soluble, using
the discussion in \S \ref{seccaspairing} we have a global solution $\mathfrak{q}_i\coloneqq (\Gamma_j^*:\Gamma_k^*:1)$
to $H_i(\Gamma_j: \Gamma_k: 1)=0$ over $K(E[2])$
for $1\leq i\leq 3$, where $H_i$ are as
in equation \eqref{hi}. Since $H_i$ is defined over $K(e_i)$ in $U_1$, $U_2$,
$U_3$ coordinates, $\Gamma_j^*$ and $\Gamma_k^*$ have $U_1,U_2,U_3$ coordinates
in $K(e_i)$. Therefore, 
we assume $\Gamma_j^*$ and $\Gamma_k^*$ to be conjugates
over $K(e_i)$,
if $e_j$ and $e_k$ are. 
Define the quantities
\begin{equation}\label{pij}
	p_{jk}\coloneqq \sqrt{\beta_j}\Gamma_j^*+\sqrt{\beta_k}\Gamma_k^*,\ \ \ \
    \mathrm{and}\ \ \ \ \ p_i\coloneqq p_{ij}p_{ik}.
\end{equation}

\begin{remark}\label{pijorbs}
	Let $\mathfrak{q}_i$ be as above.
	Let $\sigma\in G_K$ be such that $\sigma(e_i)=e_k$, then
	we can assume that the solution of the conic $H_k=0$ is
	$\mathfrak{q}_k\coloneqq \sigma(\mathfrak{q}_i)$.
	Writing $\sigma\in G_{K(e_i)}$ as
	$\sigma_s\sigma_p$,
	 where
	 $\sigma_s\in G_{K(E[2])}$ is such
	 that $\cs=\chi(\sigma_s)$,
	 and $\chi(\sigma_p)=\cv{0}$, we have
	 $$\sigma(p_{ij})=\sigma_s(p_{\sigma\cdot i,\sigma\cdot j})\ \ \text{ and }\ \
	 \sigma(p_i)=\sigma_s\left(\prod\limits_{l\ne i}
	p_{\sigma\cdot i,\sigma\cdot l}\right)=\sigma_s(p_{\sigma\cdot i}),$$
	where for indices $j,k$,
	$\sigma\cdot j=k$ if $\sigma(e_j)=e_k$.
\end{remark}
Note that there is a slight ambiguity in the above notation. The $\Gamma_i^*$ that appears
in the global point $\mathfrak{q}_j$ on $H_j=0$ may not be equal to
$\Gamma_i^*$ that appears in the
global point $\mathfrak{q}_k$ on $H_k=0$, but this will be clear from the context.

Now we resume the computation of $\eps\in \cC^2(K)$ such that $\partial\eps=\eta$.
If
$\eps\in \cC^{2}(K)$
 depends only $\ct$, $\chi'(\tau)$ and $\chi'(\rho)$, then 
$$(\partial\eps)(\sigma,\tau,\rho)=\frac{\sigma(\eps(\ct,\chi'(\tau),\chi'(\rho)))\ \eps(\cs,\chi'(\sigma),\chi'(\tau\rho))}{\eps(\cs\ct,\chi'(\sigma\tau),\chi'(\rho))\ \eps(\cs,\chi'(\sigma),\chi'(\tau))},$$
depends only on $\cs,\ct,\chi'(\sigma),\chi'(\tau)$ and $\chi'(\rho)$.
Here and henceforth we use $\eps(\tau,\rho)$ interchangeably with
$\eps(\ct,\chi'(\tau),\chi'(\rho))$.
We have the following proposition:
\begin{proposition}\label{spcaseeps}
	Let $\eps\in \cC^2(K)$ be as follows:
	\begin{equation}\label{epsilon1}
        \eps(\tau,\rho)=\begin{cases}
                &1 \ \ \ \ \ \hfill\substack{\ct=\cv{0},\chi'(\tau)\chi'(\rho)=-1\\ \mathrm{or}\ \ct=\cv{1},\chi'(\tau)\chi'(\rho)=1\\ \mathrm{or}\ \chi'(\rho)=1},\\
                \vspace{0.1in}
                &p_1\ \ \ \ \ \hfill\ct=\cv{1},\ \chi'(\tau)=1,\ \chi'(\rho)=-1,\\
                \vspace{0.1in}
                &1/p_1\ \ \ \ \ \hfill\ct=\cv{0},\ \chi'(\tau)=-1,\ \chi'(\rho)=-1,\\
                \vspace{0.1in}
                &p_{1j}\ \ \ \ \ \hfill\ct=\cv{j},\ \chi'(\tau)=1,\ \chi'(\rho)=-1,\\
                \vspace{0.1in}
                &1/p_{1j}\ \ \ \ \ \hfill\ct=\cv{k},\ \chi'(\tau)=-1,\ \chi'(\rho)=-1,
        \end{cases}
\end{equation}
	where $p_{ij}$, $p_i$ are as defined in equation \eqref{pij},
        then
        $\partial\eps=\eta$.
\end{proposition}
\begin{proof}
	If $\chi'(\rho)=1$, then $$\partial\eps(\sigma,\tau,\rho)=\frac{\sigma\eps(\ct,\chi'(\tau),1)\ \eps(\cs,\chi'(\sigma),\chi'(\tau))}{\eps(\chi(\sigma\tau),\chi'(\sigma\tau),1)\ \eps(\cs,\chi'(\sigma),\chi'(\tau))}=1=\eta(\sigma,\tau,\rho).$$
        Therefore, we assume that $\chi'(\rho)=-1$ and
	observe that
    \begin{equation}\label{ratioeqn}p_1=\frac{1}{\eps(\cv{0},-1,-1)}=\frac{\eps(\ct,1,-1)}{\eps(\ct,-1,-1)},
        \end{equation}
        for all values of $\ct$.
        Using this for $\chi'(\rho)=-1$ and $\chi'(\sigma)=1$ we have
	$$\restr{\partial\eps}{\chi'(\tau)=-1}=\frac{\sigma\eps(\ct,-1,-1)\ \eps(\cs,1,1)}{\eps(\chi(\sigma\tau),-1,-1)\ \eps(\cs,1,-1)}=(\partial\mathfrak{a}\cup\mathfrak{a}')(\sigma,\tau)\Gamma(\sigma),$$
	where $$\Gamma(\sigma)\coloneqq \frac{p_1}{\sigma(p_1)\ \eps(\cs,1,-1)^2}.$$
	Remark \ref{pijorbs} implies that $\sigma(p_1)=\sigma_s(t_1)$ and hence $\Gamma(\sigma)$
	depends only on $\chi(\sigma)$. Therefore, it is
	enough to show that $\Gamma(\sigma)=
        1/\mathfrak{a}\cup\partial\mathfrak{a}'(\cs,-1,-1),$
	for $\sigma\in G_{K(E[2])}$
     (see Appendix (Table \ref{table1})
	for explicit verification).

	Now we show that
	$\restr{\partial\eps}{\chi'(\sigma)=1}=\restr{\partial\eps}{\chi'(\sigma)=-1}.$
If $\chi'(\sigma)=-1$
	and $\chi'(\tau)=1$, then we have
	\begin{align*}
		\partial\eps(\sigma,\tau,\rho)=\frac{\sigma\eps(\ct,1,-1)\ \eps(\cs,-1,-1)}{\eps(\chi(\sigma\tau),-1,-1)\ \eps(\cs,-1,1)}&=\frac{\sigma\eps(\ct,1,-1)\ \eps(\cs,1,-1)}{\eps(\chi(\sigma\tau),1,-1)\ \eps(\cs,-1,1)}\\
		&=\restr{\partial\eps}{\chi'(\sigma)=1,\chi'(\tau)=1}(\sigma,\tau,\rho)\tag{using \eqref{ratioeqn}}.
\end{align*}
	If $\chi'(\tau)=\chi'(\sigma)=-1$, then we have
        \begin{align*}
		\partial\eps(\sigma,\tau,\rho)=\frac{\sigma\eps(\ct,-1,-1)}
		{\eps(\chi(\sigma\tau),1,-1)\ \eps(\cs,-1,-1)}
		&=\frac{\sigma\eps(\ct,-1,-1)}{\eps(\chi(\sigma\tau),-1,-1)\ \eps(\cs,1,-1)}\\
		&=\restr{\partial\eps}{\chi'(\sigma)=1,\chi'(\tau)=-1}(\sigma,\tau,\rho)\tag{using \eqref{ratioeqn}}.
\end{align*}

What is left now to show is that if
 $\chi'(\sigma)=1$, then
\begin{equation}\label{epseqeta}
\restr{\partial\eps}{\chi'(\tau)=1}(\sigma,\tau,\rho)
                        =\restr{\eta}{\chi'(\tau)=1}(\sigma,\tau,\rho)=\partial\mathfrak{a}\cup\mathfrak{a}'(\cs,\ct,-1).
 \end{equation}
       We observe that
	\begin{equation}\label{act_sig_eps}
		\sigma(\eps(\ct,\chi'(\tau),\chi'(\rho)))=
	\sigma_s\eps(\sigma\cdot\ct,\chi'(\tau),\chi'(\rho)),
	\end{equation}
	Remark \ref{pijorbs} implies \eqref{act_sig_eps}
as the values of
	$\eps$ are multiplicative combinations of $p_{1j}$ and
	$\sigma(p_{1j})=\sigma_s(p_{1\sigma\cdot j})$.
	This implies (assuming $\chi'(\sigma)=1$):
	\begin{align*}
	\restr{\partial\eps(\sigma,\tau,\rho)}{\chi'(\tau)=1,\chi'(\rho)=-1}&=
	\frac{\sigma\eps(\chi(\tau),1,-1)\eps(\chi(\sigma),1,-1)}
	{\eps(\chi(\sigma\tau),1,-1)}\\
	&=\frac{\sigma_s\eps(\sigma\cdot\chi(\tau),1,-1)\eps(\chi(\sigma),1,-1)}
	{\eps(\chi(\sigma)\sigma\cdot\chi(\tau),1,-1)}.
	\end{align*}
	Therefore, for  $\sigma,\tau\in G_K$ such that $\chi(\sigma)=\cv{i}$ and
	$\chi(\tau)=\sigma^{-1}\cdot\cv{j}$ respectively,
	$\restr{\partial\eps}{\chi'(\tau)=1,\chi'(\rho)=-1}(\sigma,\tau,\rho)$ takes the
	same value, which is similar to
	$\partial \mathfrak{a}\cup\mathfrak{a}'$.
	Hence it is enough to
	verify equation \eqref{epseqeta} assuming $\sigma,\tau\in G_{K(E[2])}$
along with  $\chi'(\sigma)=1=\chi'(\tau)=1$ and $\chi'(\rho)=-1$ (see Appendix
(Table \ref{table2}) for explicit verification).
\hfill
\end{proof}

The next subsection is dedicated to the local part of the computation of CTP using
$\eps$ obtained from
the global part.

\subsection{Local computation}\label{localcomputation}

We recall the assumption that $T_1$ (hence $\beta'$) is defined over $K_v$.
Using the local triviality of $\alpha$, for
each place $v$ of $K$
there exists a $P_v\coloneqq (x_v,y_v)\in E(\overline{K_v})$ such that
$\partial P_v(\sigma)=(\sigma-1)P_v=\alpha_v(\sigma)$ for all $\sigma\in
G_{K_v}$.
This implies that  for $\sigma\in G_{K_v}$,
$\sigma(P_v)=P_v$ if $\cs=\cv{0}$ and $\sigma(P_v)=P_v+T_i$ when $\cs=\cv{i}$, and hence
$P_v$ lies in a subfield  of $K_v(\sqrt{\beta_1},\sqrt{\beta_2})$.
Lifting $P_v$ to a degree zero divisor $\mathfrak{b}_v=(P_v)-(T_0)$, we have
\begin{equation}
(\mathfrak{a}_v-\partial\mathfrak{b}_v)(\sigma) =\begin{cases}
& 0=\div(1), \hfill \cs=\cv{0}\\
& (T_i)-(P_v+T_i)+(P_v)-(T_0)=\div\left(\frac{y-\frac{y_v(x-e_i)}{x_v-e_i}}{x-x(P_v+T_i)}\right),\  \hfill \cs=\cv{i}.\\
\end{cases}
\end{equation}
For $1\leq i\leq 3$, let $x_{v,i}$, $\theta_{v,i}$ and $\omega_{v,i}$ denote the
quantities $x_v-e_i$, $\frac{y_v}{x_v-e_i}$ and
$-\theta_{v,i}/p_{jk}$ respectively.
This gives:
\begin{equation}
(\mathfrak{a}_v-\partial\mathfrak{b}_v)\cup_1 \mathfrak{a}_{v}'(\tau,\rho) =\begin{cases}
	& 1 \hfill \substack{\ct=\cv{0}\ \mathrm{or}\ \chi'(\rho)=1},\\
\vspace{.1in}
	&z_{v,11} \ \ \ \ \ \  \hfill \substack{\ct=\cv{1},\ \chi'(\rho)=-1}\\
\vspace{.1in}
	& z_{v,1j}\ \ \ \ \ \  \hfill \substack{\ct=\cv{j},\ \chi'(\rho)=-1},
\end{cases}
\end{equation}
where using equations \eqref{xtieiandy} and \eqref{ypxpij} we have
\begin{align*}
        z_{v,11}&=\left(\frac{\left(y-\frac{y_v(x-e_1)}{x_v-e_1}\right)(-\frac{y}{x-e_1})}
        {x-x(P_v+T_1)}\right)(T_1)\times\left(\frac{x-x(P_v+T_1)}
        {\left(y-\frac{y_v(x-e_1)}{x_v-e_1}\right)(x/y)}\right)\left(T_0\right)=x_{v,1}
\end{align*}
and
\begin{align*}
        z_{v,1j} &=\left(\frac{\left(y-\frac{y_v(x-e_j)}{x_v-e_j}\right)}{x-x(P_v+T_j)}\right)(T_1)\times\left(\frac{x-x(P_v+T_j)}{\left(y-\frac{y_v(x-e_j)}{x_v-e_j}\right)(x/y)}\right)\left(T_0\right)=-\theta_k
\end{align*}
Further,
\begin{equation}
	\mathfrak{b}_v\cup_2\partial\mathfrak{a}'_{v}(\tau,\rho)=\begin{cases}
	&1 \hfill\chi'(\tau)=1\ \mathrm{or}\ \chi'(\rho)=1,\\
\vspace{.1in}
	&x_{v,1}    \ \ \ \ \ \ \ \ \hfill \chi'(\tau)=\chi'(\rho)=-1,
\end{cases}
\end{equation}
and $\gamma_v\coloneqq (\mathfrak{a}_v-\partial\mathfrak{b}_v)\cup \mathfrak{a}_{v}'-\mathfrak{b}_v\cup\partial\mathfrak{a}'_{v}-\eps_v$ is given by:
\begin{equation}\label{gamma1}
        \gamma_{v}\coloneqq  \begin{cases}
                &1 \ \ \ \ \ \hfill\substack{\ct=\cv{0},\chi'_1(\tau)\chi'_1(\rho)=-1\\ \mathrm{or}\ \ct=\cv{1},\chi'_1(\tau)\chi'_1(\rho)=1\\ \mathrm{or}\ \chi'_1(\rho)=1},\\
                \vspace{0.1in}
                &x_{v,1}/p_1\ \ \ \ \ \hfill\ct=\cv{1},\ \chi'_1(\tau)=1,\ \chi'_1(\rho)=-1,\\
                \vspace{0.1in}
                &p_1/x_{v,1}\ \ \ \ \ \hfill\ct=\cv{0},\ \chi'_1(\tau)=-1,\ \chi'_1(\rho)=-1,\\
                \vspace{0.1in}
		&\omega_{v,k}\ \ \ \ \ \hfill\ct=\cv{j},\ \chi'_1(\tau)=1,\ \chi'_1(\rho)=-1,\\
                \vspace{0.1in}
		&1/\omega_{v,k}\ \ \ \ \ \hfill\ct=\cv{k},\ \chi'_1(\tau)=-1,\ \chi'_1(\rho)=-1,
 \end{cases}
\end{equation}
where $j,k\ne 1$
are distinct.

We discuss some properties of $\omega_{v,i}$ in order to
determine the class $c_v$ in $\Br(K_v)$ represented by $\gamma_v$.
Here we digress
from the assumption that $e_1\in K_v$.
Note that $\omega_{v,i}$ 
can be defined independently of this assumption.
For $\sigma\in G_{K_v}$, satisfying $\cs=\cv{i}$ and
$\restr{\sigma}{K_v(\sqrt{\beta_i})}=\id$, and
using equation \eqref{xtieiandy} we have
\begin{equation*}
        \sigma(\omega_{v,i})=\sigma\left(\frac{-y_v}{x_{v,i}p_{jk}}\right) = \frac{y(P_v+T_i)}{p_{jk}(x(P_v+T_i)-e_i)} =\frac{-y_v}{x_{v,i}p_{jk}}=\omega_{v,i}.\tag{$\because\sigma(p_{jk})=-\sigma_sp_{\sigma\cdot j,\sigma\cdot k}=-p_{jk}$.}
\end{equation*}
This implies that $\omega_{v,i}\in K_v(\sqrt{\beta_i})$.
Further, if $\sigma\in G_{K_v}$, with $\cs=\cv{j}$ and
$\restr{\sigma}{K_v(\beta_i)}=\id$
(i.e. $\restr{\sigma}{K_v(\sqrt{\beta_i})}$ is the non-trivial element 
of $\Gal(K_v(\sqrt{\beta_i})/K_v(\beta_i))$), using equation \eqref{xpplustj-ei} we have
\begin{equation*}
   \sigma(\omega_{v,i})=\sigma\left(\frac{-y_v}{x_{v,i}p_{jk}}\right) = 
    \frac{-y(P_v+T_j)}{\sigma(p_{jk})(x(P_v+T_j)-e_i)}=
    \frac{-x_{v,i}p_{jk}}{y_v}=\frac{1}{\omega_{v,i}}. 
    \tag{$\because\sigma(p_{jk})p_{jk}=s_{kj}$ (Remark \ref{pijorbs})}
\end{equation*}
Assuming $\beta_i$ is not a
square in $K_v(\beta_i)$, we have $\mathrm{Norm}_{K_v(\sqrt{\beta_i})/K_v(\beta_i)}(\omega_{v,i})=1$.
Also, if $\sigma\in G_{K_v}$ is such
that $\cs=\cv{0}$ then we have:
$\sigma(\omega_{v,i})=\omega_{v,j}$, if
$\sigma(e_i)=e_j$.
Using Hilbert's Theorem 90, we have a $h_{v,i}\in K_v(\sqrt{\beta_i})$, such that
$\omega_{v,i}=\overline{h_{v,i}}/h_{v,i}$, where $x\mapsto\overline{x}$, represents
the non-trivial automorphism of $K_v(\sqrt{\beta_i})$ over $K_v(\beta_i)$.
We choose $h_{v,i}$ to be $1+\overline{\omega_{v,i}}\in K_v(\sqrt{\beta_i})$,
and define:
\begin{equation}\label{delta}
	\delta_{v,i}\coloneqq h_{v,i}\overline{h_{v,i}}=2+\overline{\omega_{v,i}}+
	\omega_{v,i}\in K_v(\beta_i)^{{\times}}.
\end{equation}
In the view of the above
we have the following remark:
\begin{remark}\label{omegai}
    Let $\delta_{v,i}\in K_v(e_i)^\times$ be as above.
    Then $\sigma(\delta_{v,i})=\delta_{v,j}$
     if $\sigma(e_i)=e_j$ for
     $\sigma\in G_{K_v}$.
     Therefore,
      we get:
    $\prod\limits_{i=1}^3\delta_{v,i}\in K_v^\times$
    and
    $\delta_{v,i}'\coloneqq \delta_{v,j}\delta_{v,k}\in K_v(e_i)^\times$.
    \end{remark}

    We return to our
    discussion in the case when $e_1\in K_v$, we have:
$\delta_{v,1}'\in K_v^{\times}$.
We shift $\gamma_{v}$ by the coboundary $\partial\xi_{v}$ where:
\begin{equation}
        \xi_{v}(\tau) =\begin{cases}
                &1 \  \ \ \ \ \ \ \hfill \chi'(\tau)=1,\\
                \vspace{.1in}
                & \overline{h_{v,2}h_{v,3}} \ \ \ \ \ \ \hfill \chi'(\tau)=-1.
        \end{cases}
\end{equation}
If $\gamma'_{v}\coloneqq \gamma_{v}-\partial\xi_{v}$, then we have:
\begin{equation}
        \gamma'_{v}(\tau,\rho)=\begin{cases}
                &1\ \ \ \ \ \ \ \hfill \chi'(\tau)=1\ \mathrm{or}\ \chi'(\rho)=1,\\
                \vspace{.1in}
		&\frac{1}{\delta_{v,1}'}\in K_v^\times \ \ \ \ \ \ \hfill\chi'(\tau)=\chi'(\rho)=-1,
        \end{cases}
\end{equation}
and that $\gamma'_{v}$ also represents the class
$c_{v}\in\Br(K_v)$. The following proposition implies that $c_{v}$ is the class of
the quaternion algebra $(\delta'_{v,1},\beta')$ and therefore
	$(-1)^{2\inv_{K_v}(c_{v})}=(\delta_{v,1}',\beta')_{K_v}$.

\begin{proposition}\label{classofquaternion}
If $d_1,d_2\in K_v^\times$, then the 
2-cocycle $z$ given by $(\sigma,\tau)\mapsto 1$ if $\sigma(\sqrt{d_2})=\sqrt{d_2}$ or $\tau(\sqrt{d_2})=\sqrt{d_2}$, and 
$(\sigma,\tau)\mapsto d_1$, if $\sigma(\sqrt{d_2})=\tau(\sqrt{d_2})=-\sqrt{d_2}$
    represents the class of the quaternion algebra $(d_1,d_2)$ in $\Br(K_v)$.
\end{proposition}
\begin{proof}
    \cite[\S XIV.2, Proposition 5]{locfldserre} implies that 
    the cocycle 
    \[
    x(\sigma,\tau)\coloneqq \begin{cases}
        &1 \ \ \ \ \hfill \text{if }\sigma(\sqrt{d_1})=\sqrt{d_1}\text{ or }\tau(\sqrt{d_2})=\sqrt{d_2},\\
        &-1 \ \ \ \ \hfill\text{otherwise}.
    \end{cases}
    \]
    represents the class of quaternion algebra $(d_1,d_2)$. Now it can be checked that $z=x-\partial y$ where
    \[
        y(\sigma)\coloneqq \begin{cases}
        &1 \hfill \ \ \ \ \hfill \sigma(\sqrt{d_2})=\sqrt{d_2}\\
        &1/\sqrt{d_1} \ \ \ \hfill\sigma(\sqrt{d_2})=-\sqrt{d_2}.
        \end{cases}
    \]
\end{proof}
We now express $\delta_{v,i}$ in terms of $x(Q_v)$, and
$y(Q_v)$, where $Q_v\coloneqq 2P_v\in E(K_v)$.
\begin{align*}
        x(Q_v)-e_i &= \left(\frac{3x_v^2+c}{2y_v}\right)- (x_v-e_j)-(x_v-e_k)\\
                &= \frac{1}{4}\left(\sum\limits_{i=1}^3\theta_{v,i}\right)^2-\theta_{v,i}\theta_{v,k}-\theta_{v,i}\theta_{v,j}= \frac{1}{4}\left(\theta_{v,j}+\theta_{v,k}-\theta_{v,i}\right)^2.
\end{align*}
There exists $w_{v,i}\in K_v(e_i)^\times$, such that $x(Q_v)-e_i= \beta_iw_{v,i}^2$,
hence we have $\theta_{v,i}=w_{v,j}\sqrt{\beta_j}+w_{v,k}\sqrt{\beta_k}$ and
$$\omega_{v,i}=-\frac{\theta_{v,i}}{p_{jk}}=-\frac{w_{v,j}\sqrt{\beta_j}+w_{v,k}\sqrt{\beta_k}}{\Gamma_j^*\sqrt{\beta_j}+\Gamma_k^*\sqrt{\beta_i}}.$$
Here
$w_{v,i}$ are chosen to be conjugates over
$K_v$ if $e_i$
 are.
Therefore,
\begin{equation}\label{deltavi}
        \delta_{v,i}= 2\left(1-\frac{\beta_jw_{v,j}\Gamma_j^*-\beta_kw_{v,k}\Gamma_k^*}{\beta_j(\Gamma_j^*)^2-\beta_k(\Gamma_k^*)^2}\right)=2\left(1+\frac{\beta_kw_{v,k}\Gamma_k^*-\beta_jw_{v,j}\Gamma_j^*}{s_{kj}}\right).
\end{equation}
A value of $\langle a, a'\rangle_\ctp$ (depending on the choices made above) is then given by the following theorem.
\begin{theorem}\label{condvalctp}
In the view of the above discussion and the choice of
the point $P_v$ made above, the value of CTP on $(a,a')\in\Sel^2(E)\times
    \H^1(G_K, \langle T_1\rangle)$ is equal to:
$$\langle a,a'\rangle_\ctp=\prod\limits_v(\delta_{v,i}',\beta')_{K_v}.$$
\end{theorem}


\section{Computing CTP on $\Sel^2(E)\times\Sel^2(E)$}
Our main aim in this section is to prove the sufficiency of the
computation done in the previous section.
\subsection{The corestriction method}\label{cor_method}
Let $a'\in \Sel^{(2)}(E)$ be represented by the 1-cocycle $\alpha'$
which corresponds to the triple $(\beta_1',\beta_2',\beta_3')$ as in section
\ref{notations},
and we drop the subscript in $\chi_{a'}$ and call it $\chi'$. Note
that this $\chi'$
is not the same as the one in the previous section.
We choose a lift of
$\alpha'$ to $\cC^1(G_K,\Div^0(E))$ as 
\[
\mathfrak{a}'(\sigma)=\begin{cases}
	&0 \ \ \ \ \ \ \ \ \ \hfill\csd=\cv{0},\\
	\vspace{.1in}
	&(T_j)+(T_k)-2(T_0) \ \ \ \ \ \ \csd=\cv{i}.
\end{cases}
\]
Note that the choice of lift $\frak{a}'$ of $\alpha'$ is different from the choice of lift
$\frak{a}$ of $\alpha$ in \S \ref{GlobComp}.
The following lemma implies that
$\mathfrak{a}'$ can be written  as a sum of corestrictions
of certain cochains.
\begin{lemma}\label{corestrionlemma}
Let $\mathfrak{a}'$ be as
above, and let
$\Delta_1,\ldots,\Delta_n$ be the different orbits of $\Delta$ with representatives 
    $T_1,\ldots, T_n$ for
$n\leq 3$. Then we have:
$$\mathfrak{a}'=\sum\limits_{i=1}^n\cor(\mathfrak{t_i}),$$
where $t_i\in\cC^1(G_{K(e_i)},\Div^0(E))$ is given by:
\[
    \mathfrak{t}_i(\sigma)\coloneqq \begin{cases}
    &0 \  \ \ \ \hfill \sigma(\sqrt{\beta_i'})=\sqrt{\beta_i'},\\
    &(T_i)-(T_0)\  \ \ \ \hfill \sigma(\sqrt{\beta_i'})=-\sqrt{\beta_i'},
    \end{cases}
\]
and the corestriction of $\mathfrak{t}_i$ is taken with respect to the groups
    $G_{K(e_i)}$ and $G_{K}$.
\end{lemma}
\begin{proof}

Let $T_i=P_{i,1},\ldots P_{i,k_i}$ be the
points in the orbit of $T_i$ and let $\beta_i'=b_{i,1},\ldots, b_{i,k_i}$ be the
    $G_K$-conjugates of $\beta_i'$ in $\{\beta_1',\beta_2',\beta_3'\}$.
Let $\{\id=\tau_{i,1},\tau_{i,2},\ldots, \tau_{i,k_i}\}$ be the
representatives of the
right cosets of $G_{K(e_i)}$ with $\tau_{i,j}(P_{i,j})=T_i$ for
$1\leq j\leq k_i$, and
$\tau_{i,j}\cdot\left(\sum\limits_{l=1}^{k_i}\sqrt{b_{i,l}}(P_{i,l})\right)=\sum\limits_{l=1}^{k_i}\sqrt{b_{i,l}}(P_{i,l})$.
To see that such a choice of coset representatives is possible we note
that any $\sigma\in G_K$ such that $\sigma(P)=T_i$ for some $P\in \Delta_i$ has the form
$\sigma_s\sigma_p$ (as in Remark \ref{pijorbs}), and so $\sigma_p$ has the required property.
Let $m,r$ be the maps as in the definition of the corestriction map.
Then $m(\tau_{i,j}\sigma)\in G_{K(e_i)}$ and hence $r(\tau_{i,j}\sigma)^{-1}(\sqrt{\beta_i'})=\sqrt{b_{i,\sigma^{-1}\cdot j}}$, where $(i,\sigma^{-1}\cdot j)=(i,l)$ if $\sigma^{-1}(P_{i,j})=P_{i,l}.$
Using equation \eqref{corbarres} and the definition of $\mathfrak{t}_i$ we have:
\begin{align*}
    \sum\limits_{i=1}^n\cor^{G_K}_{G_{K(e_i)}}(\mathfrak{t}_i)(\sigma)&=\sum\limits_{i=1}^n\sum\limits_{j=1}^{k_i}\tau_{i,j}^{-1}\mathfrak{t}_i(m(\tau_{i,j}\sigma))\\
    &=\sum\limits_{i=1}^n\sum\limits_{j=1}^{k_i} g\left(\frac{\sigma (\sqrt{b_{i,\sigma^{-1\cdot j}})}}{\sqrt{b_{i,j}}}\right)\left((P_{i,j})-(T_0)\right)
\end{align*}
where $g$ is the isomorphism $g:\mu_2\to \Z/2\Z$.
Now using the definition of $\chi'$ we have $\chi'(\sigma)=\sum\limits_{i=1}^n\sum\limits_{j=1}^{k_i}\chi_{i,j}'(\sigma)(P_{i,j})$,
where
$\chi'_{i,j}(\sigma)\coloneqq \frac{\sigma(\sqrt{b_{i,\sigma^{-1}\cdot j}})}{\sqrt{b_{i,j}}}$
 denotes the value of $\chi'(\sigma)$ at $P_{i,j}$.
Therefore,
$\sum\limits_{i=1}^n\cor^{G_K}_{G_{K(e_i)}}(\mathfrak{t}_i)=\mathfrak{a}'$.
\end{proof}
The above lemma along with Proposition \ref{rescor} part (3)
immediately gives us the following
\begin{corollary}
Let $a,a'\in\Sel^{(2)}(E)$ and $\alpha$, $\alpha'$, $\mathfrak{a}$ and $\mathfrak{a}'$ 
    be as in the definition of CTP. Assume the notations of Lemma \ref{corestrionlemma} 
    and that $\mathfrak{a}'$ is chosen as in Lemma \ref{corestrionlemma}. Then we have:
$\eta\coloneqq \partial\mathfrak{a}\cup_1\mathfrak{a}'-
    \mathfrak{a}\cup_2\partial\mathfrak{a}'=\sum\limits_{i=1}^n\cor^{G_K}_{G_{K(e_i)}}\eta_i$, 
    where
$$\eta_i\coloneqq \partial\res^{G_{K(e_i)}}_{G_K}(\mathfrak{a})\cup \mathfrak{t}_i-
    \res^{G_{K(e_i)}}_{G_K}(\mathfrak{a})\cup\partial\mathfrak{t}_i\in\zz^3(K(e_i)).$$
In particular, if $\eps_i\in \cC^2(K(e_i))$ are such that $\partial\eps_i=\eta_i$, 
    then $\eps\in\cC^2(K)$ such that $\partial\eps=\eta$ can be chosen to be 
    $\sum\limits_{i=1}^n\cor^{G_K}_{G_{K(e_i)}}(\eps_i)$.
\end{corollary}

Thus we reduce the case of computing a suitable $\eps$ for
$a,a' \in \Sel^{(2)}(E)$, to the case of
computing $\eps_i$ which we have
already done in Proposition
\ref{spcaseeps} by
setting $K$ as $K(e_i)$ and $T_1$ as $T_i$.

Considering the
local part of the computation we have:
$\gamma_{v}=\sum\limits_{i=1}^n\gamma_{i,v}'$
where
$$\gamma_{i,v}'\coloneqq
(\mathfrak{a}_v-\partial\mathfrak{b}_v)\cup\left(\cor^{G_K}_{G_{K(e_i)}}(\mathfrak{t}_i')\right)_v-\mathfrak{b}_v\cup\partial\left(\cor^{G_K}_{G_{K(e_i)}}(\mathfrak{t}_i')\right)_v-\left(\cor^{G_K}_{G_{K(e_i)}}(\eps_i)\right)_v.$$
Using the double coset formula (equation
\eqref{doublecosetfornumfields}) we have:
$$\left(\cor^{G_K}_{G_{K(e_i)}}(\mathfrak{t}_i)\right)_v=\sum\limits_{w|v}\cor^{G_{K_v}}_{G_{K(e_i)_w}}\mathfrak{t}_{i,w},$$
where $\mathfrak{t}_{i,w}\coloneqq
\res^{G_{K(e_1)_w}}_{G_{K(e_{i,w})}}((g_{i,w})_*\mathfrak{t}_i)\in \cC^1(G_{K(e_i)_w}, \langle(T_{i,w})-(T_0)\rangle)$,
$g_{i,w}\in G_K$
corresponds to the valuation $w$ of $K(e_i)$ above $v$ and
$e_{i,w}$, $T_{i,w}$  are the $g_{i,w}$-conjugates of
$e_i$, $T_i$ respectively.
Concretely, we have:
\[
\mathfrak{t}_{i,w}(\sigma)\coloneqq \begin{cases}
&0 \  \ \ \ \hfill \sigma(\sqrt{\beta_{i,w}'})=\sqrt{\beta_{i,w}'},\\
    &(T_{i,w})-(T_0)\  \ \ \ \hfill \sigma(\sqrt{\beta_{i,w}'})=-\sqrt{\beta_{i,w}'},
\end{cases}
\]
where $\beta_{i,w}'\coloneqq g_{i,w}(\beta_i')$.

Similarly, applying the
double coset formula for
$\eps_i$ we get:
$$\left(\cor^{G_K}_{G_{K(e_i)}}\eps_i\right)_v=\sum\limits_{w|v}\cor^{G_{K_v}}_{G_{K(e_i)_w}}\eps_{i,w}.$$
Hence, $\gamma_{i,v}'=\sum\limits_{w|v}\cor^{G_{K_v}}_{G_{K(e_i)_w}}\gamma_{i,w},$
where
\begin{equation}
    \gamma_{i,w}\coloneqq
    \left(\res^{G_{K(e_i)_w}}_{G_K}(\mathfrak{a})-\res^{G_{K(e_i)_w}}_{G_{K_v}}(\partial\mathfrak{b}_v)\right)\cup
    \mathfrak{t}_{i,w}-\res_{G_{K_v}}^{G_{K(e_i)_w}}(\mathfrak{b}_v)\cup\partial\mathfrak{t}_{i,w}-\eps_{i,w}.
\end{equation}
The following proposition shows that $\gamma_{i,w}$ is a 2-cocycle.
\begin{proposition}
$\gamma_{i,w}\in \cZ^2(K(e_i)_w)$.
\end{proposition}
\begin{proof}
 Using $\partial\eps_i=\eta_i$ we have
    \begin{align*}
     \partial\gamma_{i,w}
     &=\res^{G_{(e_i)_w}}_{G_K}\partial\frak{a}\cup
        \frak{t}_{i,w}'-\res^{G_{K(e_i)_w}}_{G_K}\frak{a}\cup
        \partial\frak{t}_{i,w}'\\
        &\ \ \  -\res_{G_{K(e_{i,w})}}^{G_{K(e_i)_w}}(g_{i,w})_*
        \left(\partial\res_{G_K}^{G_{K(e_i)}}\frak{a}\cup\frak{t}_i'-\res_{G_K}^{G_{K(e_i)}}\frak{a}\cup\partial\frak{t}_i'\right)\tag{using
        the double coset formula (equation \eqref{doublecosetfornumfields}) on
        $\eta_i$}\\
     &=\res^{G_{K(e_i)_w}}_{G_K}(\partial(\frak{a}-(g_{i,w})_*\frak{a}))\cup\frak{t}_{i,w}'-\res^{G_{K(e_i)_w}}_{G_K}(\frak{a}-(g_{i,w})_*\frak{a})\cup\partial\frak{t}_{i,w}'\tag{$\because(g_{i,w})_*$ commutes with $\res$, $\cup$ and $\partial$}.
 \end{align*}
 So if $(g_{i,w})_*(\frak{a})=\frak{a}$, then
 $\partial\gamma_{i,w}=0$.
 Note that $\frak{a}(\tau)$ only depends on $\ct$, therefore we can  equivalently write $\frak{a}(\ct)$
 instead of $\frak{a}(\tau)$. We have $\sigma\frak{a}(\ct)=\frak{a}(\sigma\cdot\ct)$. 
 To see this recall that 
 $\ct=\sum\limits_{i=1}^3a_i(T_i)$, 
 for some $a_i\in \mu_2$ depending on $\tau$, hence if $\ct=\cv{j}$, then $\sigma\ct=\cv{\sigma\cdot j}$.

 Now for $\sigma\in G_K$ 
 \begin{align*}
     \left((g_{i,w})_*(\frak{a})\right)(\sigma)&=g_{i,w}\frak{a}(g_{i,w}^{-1}\sigma g_{i,w})\tag{by definition}\\
     &=g_{i,w}\frak{a}(\chi(g_{i,w}^{-1}\sigma g_{i,w}))=\frak{a}(g_{i,w}\chi(g_{i,w}^{-1}\sigma g_{i,w}))\\
     &=\frak{a}(\chi(\sigma g_{i,w})\chi(g_{i,w})^{-1})=\frak{a}(\cs\sigma\chi(g_{i,w})\chi(g_{i,w})^{-1})\tag{$\chi$ is a 1-cocycle}\\
 \end{align*}
 Recall from the proof of Lemma \ref{corestrionlemma}
or from Remark \ref{pijorbs} that $g_{i,w}$ can be chosen such that
    $\chi(g_{i,w})=\cv{0}$ via the decomposition $\sigma=\sigma_s\sigma_p$ for
    $\sigma\in G_K$. Making such a choice 
for $g_{i,w}$ we have
\[ \left((g_{i,w})_*(\frak{a})\right)(\sigma)=\frak{a}(\cs\sigma\chi(g_{i,w})\chi(g_{i,w})^{-1})=\frak{a}(\cs)=\frak{a}(\sigma).\qedhere\]

\end{proof}

The above proposition together with Proposition \ref{rescor} part (2)
implies that
$$\inv_{K_v}([\gamma_v])=\sum\limits_{i=1}^n\inv_{K_v}([\gamma'_{i,v}])=\sum\limits_{i=1}^n\sum\limits_{w|v}\inv_{K_v(e_{i,w})}([\gamma_{i,w}]),$$ where $[z]$
represents the cohomology class of the cocycle $z$. This shows that the contribution
from a place $v$ of $K$ in CTP is  the sum of contributions from $G_{K_v}$-
orbits of $\Delta$.
Recall the definition of $\prod\limits_i^\diamond$ from \S \ref{notations}.
From the above
computation and the \S \ref{localcomputation} we get $\delta'_{i,v}\in K_v(e_i)^\times$ such that
the local contribution at $v$ in CTP is
$\prod\limits_{i}^\diamond(\delta'_{i,v},\beta_i')_{K_v(e_i)}.$
Therefore, we have the following theorem
\begin{theorem}\label{ctpabstformula}
We have
    $$(-1)^{2\langle
    a,a'\rangle_\ctp}=\prod\limits_{v}\prod\limits_{i}^\diamond(\delta'_{i,v},\beta_{i}')_{K_v(e_i)},$$
    where $i$ runs through the $G_{K_v}$-orbits of $\Delta$ for each place $v$ of
    $K$.
\end{theorem}
The following corollary says that we only need to consider the 
contributions to the CTP from finitely many places and that what those places are.
\begin{corollary}
Let $S_{a, a'}$ be the set of finite places such that either of $\alpha_v$, $\alpha'_v$ 
(the localizations at $v$ of cocycles representing $a$, $a'$ (resp.)) factor through a ramified 
extension. Clearly $S_{a,a'}\subset \{\text{Primes of bad reduction of }E/K\}\cup \{2\}$. Then with above notations we have
$$(-1)^{2\langle a,a'\rangle_\ctp}=\prod\limits_{v\in
    S}\prod\limits_{i}^\diamond(\delta'_{i,v},\beta_{i}')_{K_v(e_i)},$$
    where $i$ runs through $G_{K_v}$-orbits of $\Delta$.
\end{corollary}
\begin{proof}
    Let $v$ be a finite place of $K$. It is easy to see that if the localization 
    of $\alpha_v$ or $\alpha'_v$ is trivial as a cocycle
    then the contribution to the CTP at $v$ is trivial. So we assume 
    that $\alpha_v$, $\alpha_v'$
    factor through a non-trivial unramified extension. This implies that 
    $K_v(\sqrt{\beta_i})/K_v(e_i)$ is 
    an unramified extension (similarly for $\beta_i'$).
    Now if either of $\beta_i$ or $\beta_i'$ is a
    square in $K_v(e_i)$ then the local contribution corresponding to the orbit 
    of $e_i$ locally is trivial. Otherwise, $K_v(\sqrt{\beta_i})=K_v(\sqrt{\beta_i'})$ 
    over $K_v(e_i)$ and hence $\delta_{v,i}'$ is the norm of an element of 
    $K_v(\sqrt{\beta_i'})$ to $K_v(e_i)$.
    
    For an infinite place $v$, if $\beta_i$ is complex or positive,
    $\delta_{v,i}$ can be chosen to be 1. Hence if $\beta_i$ or $\beta_i'$ is 
    complex or positive then the contribution from the orbit of $e_i$ is trivial. 
    Therefore, we assume that $v$
    is a real place, then at least one of $\beta_1,\beta_2,\beta_3$ (say
    $\beta_1$) is real and
    negative. If $\beta_2$, $\beta_3$ are complex then $\delta_1'=1$ and other
    orbits correspond to complex valuation, so the contribution at $v$ is
    trivial. 
    Otherwise, $\beta_2$ and $\beta_3$ are real, and exactly one of them is
    negative. So $\delta_{v,i}'$ is a norm from $\bbC$ to $\bbR$ and hence the 
    contribution at $v$ is trivial.
    One can also use the equation \eqref{ctpexactformula} (see \S
    \ref{sec:CTPExactFormula}) along with the fact that $\delta_{v,i}$ is a norm
    of an element in $\bbC$ to conclude that the contributions at the infinite
    places are trivial. 
\end{proof}

\subsection{Exact formula for CTP}\label{sec:CTPExactFormula}
In order to obtain an exact formula
for CTP, we express the
formula for CTP obtained in the Theorem \ref{ctpabstformula} in terms of
the Hilbert symbols $(\delta_{v,i},\beta_i')_{K_v(e_i)}.$
Remark \ref{omegai} implies that there is
a $d_v\in K_v^\times$ such that $\prod\limits_{i=1}^3d_v\delta_{v,i}\in
(K_v^\times)^2.$ Previously, we expressed the CTP in terms of the Hilbert
symbols $(\delta_{v,i}',\beta_i')_{K_v(e_i)}$.
Using Theorem \ref{ctpabstformula}, if $c_v$ denotes the class of
$\gamma_v$ in $\Br(K_v)$, then we have:
\begin{align*}
(-1)^{2\inv_{K_v}(c_v)}&=\prod\limits_{i}^\diamond(\delta_{v,i}',\beta_i')_{K_v(e_i)}=
\prod\limits_{i}^\diamond(d_v^2,\beta_i')_{K_v(e_i)}(\delta_{v,i}',\beta_i')_{K_v(e_i)}\\
&=\prod\limits_{i}^\diamond(d_v^2\delta_{v,i}',\beta_i')_{K_v(e_i)}
=\prod\limits_{i}^\diamond(d_v\delta_{v,i},\beta_i')_{K_v(e_i)}\\
&=\prod\limits_i^\diamond(d_v,\beta_i')_{K_v(e_i)}\prod\limits_i^\diamond(\delta_{v,i},\beta_i')_{K_v(e_i)}.
\end{align*}
Here $\delta_{v,i}'\coloneqq \delta_{v,j}\delta_{v,k}$
 is as in Remark \ref{omegai}.
We now show that
$\prod\limits_i^\diamond(d_v,\beta_i')_{K_v(e_i)}=1$.
For this we use the fact that if $L$ is a finite
extension of $K_v$ and $z\in K_v^\times$, $z'\in L^\times$, then
$(z,z')_{L}=(z,\mathrm{Norm}_{L/K_v}(z'))_{K_v}$.
Using this we get:
$$\prod\limits_{i}^\diamond(d_v,\beta_i')_{K_v(e_i)}=\prod\limits_{i}^\diamond(d_v,\mathrm{Norm}_{K_v(e_i)/K_v}(\beta_i)')_{K_v}=(d_v,\beta_1'\beta_2'\beta_3')_{K_v}=1.$$
Therefore, we get:
\begin{equation}\label{ctpexactformula}
(-1)^{2\langle
    a,a'\rangle_\ctp}=\prod\limits_v\prod\limits_i^\diamond(\delta_{v,i},\beta_i')_{K(e_i)}.
\end{equation}
One verifies that the above equation looks similar to the expression in equation \eqref{cas}
for the Cassels' pairing.
The following theorem  shows that the
pairing $\langle\cdot,\cdot\rangle_\mathrm{Cas}$  is the same as the $\langle\cdot,\cdot\rangle_\ctp$, using
the expression for $\delta_{v,i}$ (equation \eqref{deltavi}).

\begin{theorem}\label{cas=ctp}
For $a,a'\in \Sel^{(2)}(E)$, we have
        $$\langle a,a'\rangle_{\mathrm{Cas}}=\langle a,a'\rangle_{\ctp}.$$
\end{theorem}
\begin{proof}
Let $\mathfrak{q}_i\coloneqq (\Gamma_j^*:\Gamma_k^*:1)$ be a
	global point on $H_i(\Gamma_j,\Gamma_k,T)$ (as in equation \eqref{pij}), then
\begin{align*}
        L_i&\coloneqq \sum\limits_{l=1}^3U_l\frac{\partial H_i}{\partial U_l}(\mathfrak{q}_i)+T\frac{\partial H_i}{\partial T}(\mathfrak{q}_i)\\
        &=\frac{\partial H_i}{\partial\Gamma_j}(\mathfrak{q}_i)\left(\sum\limits_{l=1}^3U_l\frac{\partial\Gamma_j}{\partial U_l}(\mathfrak{q}_i)\right)+ \frac{\partial H_i}{\partial\Gamma_k}(\mathfrak{q}_i)\left(\sum\limits_{l=1}^3U_l\frac{\partial\Gamma_k}{\partial U_l}(\mathfrak{q}_i)\right)+T\frac{\partial H_i}{\partial T}(\mathfrak{q}_i)\\
        \tag{$\because(\Gamma_1,\Gamma_2,\Gamma_3)$ is linear change of coordinates from $(U_1,U_2,U_3)$}\\
        &=\Gamma_j\frac{\partial H_i}{\partial \Gamma_j}(\mathfrak{q}_i)+\Gamma_k\frac{\partial H_i}{\partial \Gamma_k}(\mathfrak{q}_i)+T\frac{\partial H_i}{\partial T}(\mathfrak{q}_i).
\end{align*}
Let $\mathfrak{q}_v=(w_{v,1}:w_{v,2}:w_{v,3}:1)$ be the point on $D_{a}$
        (in $(\Gamma_1:\Gamma_2:\Gamma_3:T)$ coordinates) corresponding
to $Q_v$ (as in the previous section).
Then we have $L_i(q_v)=2(e_k-e_j)+2(\beta_k\Gamma_k^*w_{v,k}-\beta_j\Gamma_j^*w_{v,j})$.
Note that $(\delta_{v,i},\beta_i')_{K_v(\beta_i)}=(s_{kj}\delta_{v,i},\beta_i')_{K_v(\beta_i)}$,
since $s_{kj}$ is a norm. Further using the expression for $\delta_{v,i}$ in equation \eqref{deltavi},
and $p_{kj}=\sqrt{\beta_k}\Gamma_k^*+\sqrt{\beta_j}\Gamma_j^*$,
we have $s_{kj}\delta_{v,i}=L_i(\mathfrak{q}_v),$ hence the pairing $\langle\cdot,\cdot\rangle_{\mathrm{Cas}}$
defined in equation
\eqref{cas} is the same as the Cassels-Tate pairing.
\hfill
\end{proof}

\section{Appendix}
In this section we verify the two computational claims made in Proposition \ref{spcaseeps} i.e.
$1/\Gamma_1(\sigma)=\frak{a}\cup\partial\frak{a}'(\sigma,-1,-1)$, and 
$\partial\eps(\sigma,\tau, \rho)= 
\partial\frak{a}\cup\frak{a}'(\sigma,\tau,\rho)$, when
$\chi'(\tau)=1$ and $\chi'(\rho)=-1$.

Let $\sigma$, $\tau$, $\rho\in G_K$ be such that
$\chi'(\tau)=\chi'(\rho)=-1$, then we have:
\begin{center}
\begin{tabular}{|c|c|c|}
     \hline
     $\cs$ & $1/\Gamma_1(\sigma)$ & $\mathfrak{a}\cup\partial\mathfrak{a}'(\sigma,-1,-1)$ \\
     \hline
     $\cv{0}$ & $\frac{p_1(1)^2}{p_1}=1$ & $1$ \\
     \hline
     $\cv{1}$ & $\frac{\sigma(p_1)p_1^2}{p_1}=s_1$ & $s_1$\\
     \hline
     $\cv{j}$ & $\frac{\sigma(p_1)p_{1j}^2}{p_1}= s_{j1}$ & $s_{j1}$\\
    \hline
\end{tabular}
\captionof{table}{$1/\Gamma_1(\sigma)=\mathfrak{a}\cup\partial\mathfrak{a}'(\sigma,-1,-1)$.}\label{table1}
\end{center}

Let $\sigma$, $\tau$, $\rho\in G_K$ be such that
$\chi'(\rho)=-1$, $\chi'(\tau)=1$ and $\chi'(\sigma)=1$, then if
$\cs=\cv{0}$, then
$$
\partial\eps(\sigma,\tau,\rho)=\frac{\sigma_s\eps(\sigma\cdot\chi(\tau),1,-1)}
	{\eps(\sigma\cdot\chi(\tau),1,-1)}=1=
	\partial\mathfrak{a}\cup\mathfrak{a}'(\sigma,\tau),
$$
and if $\ct = \cv{0}$, then
$$
\partial\eps(\sigma,\tau,\rho)=\frac{\eps(\cs,1,-1)}
	{\eps(\cs,1,-1)}=1=
	\partial\mathfrak{a}\cup\mathfrak{a}'(\sigma,\tau).
$$
The following table symbolically verifies all the other possible cases:
\begin{center}
\begin{tabular}{|c|c|c|}
     \hline
$\cs$ & $\sigma\cdot\ct$ & $\restr{\partial\eps(\cs,\sigma\cdot\ct,-1)}{\chi_1'(\sigma)=\chi_1'(\tau)=1}$ \\
\hline
$\cv{1}$ & $\cv{1}$ & $\sigma_s(p_1)p_1=s_1$ \\
\hline
$\cv{j}$ & $\cv{j}$ & $\sigma_s(p_{1j})p_{1j}=s_{1j}$ \\
\hline
$\cv{j}$ & $\cv{k}$ & $\frac{\sigma_s(p_{1k})p_{1j}}{p_1}=-1$ \\
\hline
$\cv{1}$ & $\cv{k}$ & $\frac{\sigma_s(p_{1k})p_1}{p_{1j}}=s_{k1}$ \\
\hline

$\cv{k}$ & $\cv{1}$ & $\frac{\sigma_s(p_1)p_{1k}}{p_{1j}}=s_{k1}$ \\
\hline
\end{tabular}
\captionof{table}{$\restr{\partial \eps}{\chi'(\tau)=1,\chi'(\rho)=-1}(\sigma,\tau,\rho)=\partial\mathfrak{a}\cup\mathfrak{a}'(\sigma,\tau,\rho).$}\label{table2}
\end{center}

\end{document}